\documentclass[11pt,a4paper]{article}
\usepackage{amsmath}
\usepackage{amsfonts}
\usepackage{amsthm}
\usepackage{amscd}
\usepackage{amsbsy}
\usepackage{amsrefs,enumerate}
\usepackage[latin2]{inputenc}
\usepackage{t1enc}
\usepackage[mathscr]{eucal}
\usepackage{indentfirst}
\usepackage{graphicx}
\usepackage{graphics}
\usepackage{pict2e}
\usepackage{epic}
\numberwithin{equation}{section}
\usepackage[margin=2.9cm]{geometry}
\usepackage{epstopdf}
\usepackage{tikz-cd}
\usepackage{chngcntr}
\usepackage{amssymb}
\usepackage{esint}
\usepackage{hyperref}
\usepackage{mwe}
\usepackage{mathrsfs}
\usepackage{xfrac}
\usepackage{verbatim}
\usepackage{inputenc}
\usepackage{soul}
\usepackage{lipsum}
\allowdisplaybreaks
\usepackage{bbm}

\theoremstyle{plain}
\newtheorem{Th}{Theorem}[section]
\newtheorem{Lemma}[Th]{Lemma}

\theoremstyle{definition}
\newtheorem{Def}[Th]{Definition}

\newtheorem{Rem}[Th]{Remark}
\newtheorem{?}[Th]{Problem}

\newcommand*\R{\mathbb{R}}
\newcommand*\C{\mathbb{C}}

\newcommand*\Om{\Omega}

\newcommand{\Div}{\text{div}_x}

\newcommand{\vectoru}{\mathbf{u}}

\newcommand{\vectorv}{\mathbf{v}}

\newcommand{\Epsilon}{\mathcal{E}}
\newcommand{\dx}{\;\!\text{d}x}
\newcommand{\dt}{\;\text{d}t\,}

\newcommand{\vectorphi}{\pmb{\varphi}}

\newcommand{\Nu}{\mathcal{V}}
\newcommand{\Nutx}{\mathcal{V}_{t,x}}
\newcommand{\DD}{\mathbb{D}}

\newcommand{\Ov}[1]{\overline{#1}}
\newcommand{\Un}[1]{\underline{#1}}
\newcommand{\vc}[1]{{\bf #1}}
\newcommand{\bb}[1]{{\mathbb #1}}
\newcommand{\vr}{\varrho}
\newcommand{\tvr}{\tilde{\varrho}}
\newcommand{\vt}{\vartheta}
\newcommand{\tvt}{\tilde{\vartheta}}
\newcommand{\tT}{\widetilde{\Theta}}
\newcommand{\tu}{\tilde{\vc{u}}}
\newcommand{\dvt}{\vc{D}_\vt}
\newcommand{\vrthu}{\vr,\vt,\vc{u}}
\newcommand{\tildevrthu}{\tvr,\tvt,\tu}
\newcommand{\relenttuple}{\vrthu\,\vert\, \tildevrthu}
\newcommand{\lel}{\left\langle}
\newcommand{\ril}{\right \rangle}
\newcommand{\levert}{\left\vert}
\newcommand{\rivert}{\right \vert}

\newcommand{\inttom}{ \int_0^ \tau \int_{ \Om} }
\newcommand{\du}{\bb{D}_{\vc{u}}}
\newcommand{\vu}{\vc{u}}
\newcommand{\vm}{\vc{m}}
\newcommand{\D}{\text{d}}
\newcommand{\intO}[1]{\int_{\Omega} #1  \dx}
\newcommand{\br}{\nonumber \\}
\begin{document}
	
	\title{On weak(measure valued)-strong uniqueness for Navier--Stokes--Fourier system with Dirichlet boundary condition}

	\author{Nilasis Chaudhuri
			\thanks{E-mail:\tt n.chaudhuri@imperial.ac.uk}
	}
	\maketitle
	\vspace{-5mm}
	
	\centerline{Imperial College London}
	\centerline{Department of Mathematics, 180 Queen's Gate, Huxley Building, }
	\centerline{ SW7 2AZ, London, United Kingdom}

	\begin{abstract}
	 In this paper our goal is to define a measure valued solution of compressible Navier--Stokes--Fourier system for a heat conducting fluid with Dirichlet boundary condition for temperature in a bounded domain. The definition is based on the weak formulation of entropy inequality and ballistic energy inequality. Moreover, we obtain the \textit{weak(measure valued)-strong uniqueness} property of this solution with the help of relative energy. 
	\end{abstract}

	
	
	\section{Introduction}
	We consider the Navier--Stokes--Fourier system in the time-space cylinder $ Q=(0,T) \times \Om $, where $ T>0  $ and $ \Om $ is a bounded domain of $ \R^d $ with $ d=2 $ or $ 3 $. The time evolution of the \emph{density} $ \varrho =\varrho(t,x) $, \emph{velocity} $ \vc{u}=\vc{u}(t,x) $ and the \emph{absolute temperature} $ \vartheta=\vartheta(t,x) $ of a compressible, viscous and heat conducting fluid is given by the following system of field equations which describes conservation of mass, momentum and internal energy:
	
	\begin{align}
			&\partial_t \varrho + \Div (\varrho \vc{u}) =0,\label{NSF-c}\\
			&\partial_t (\varrho \vc{u}) + \Div (\varrho \vc{u} \otimes \vc{u}) + \nabla_x p(\varrho,\vartheta) = \Div \bb{S} + \varrho \vc{g},\label{NSF-m}\\
			&\partial_t (\varrho e(\varrho,\vartheta)) + \Div (\varrho e(\varrho,\vartheta)\vu ) + \Div \vc{q} = \bb{S}\colon \bb{D}_x \vc{u} - p(\varrho , \vartheta) \Div \vc{u} ,\label{NSF-e}
	\end{align}
	where $\bb{D}_x \vc{u}=\frac{1}{2}( \nabla_x \vc{u} +  \nabla_x^t \vc{u})  $, $ p= p(\varrho,\vartheta) $ is the pressure and $ e=e(\varrho, \vartheta) $ is the internal energy and $ \vc{g} $ is an external force. 
	The pressure and the internal energy is interrelated by means of \emph{Gibbs' equation}
	\begin{align}\label{GE}
		\vartheta Ds= De+pD\left(\frac{1}{\varrho}\right),
	\end{align}
	where $ s=s(\varrho,\vartheta)  $ is the \emph{entropy} and $D= \begin{pmatrix}
		\frac{\partial}{\partial \vr} \\
		\frac{\partial}{\partial \vt}\\
	\end{pmatrix}.$
	We assume the fluid is Newtonian, hence the \emph{viscous stress tensor}$ (\bb{S}) $ is given by 
	\begin{align}\label{visc_str}
		\mathbb{S}(\nabla_x \mathbf{u})=\mu(\varrho , \vartheta) \bigg(\frac{\nabla_x \vectoru + \nabla_x^{T} \vectoru}{2}-\frac{1}{d} (\Div\vectoru)\mathbb{I} \bigg) + \lambda(\varrho , \vartheta) (\Div \vectoru) \mathbb{I},
	\end{align}
	where $ \mu $ is the \emph{shear viscosity} coefficient and $ \lambda $ is the \emph{bulk viscosity} coefficient with $ \mu>0 $ and $ \lambda\geq 0 $. The \emph{heat flux} $ \vc{q} $ is given by the \emph{Fourier law}, 
	\begin{align}\label{Fou_law}
		\vc{q}(\vr,\vartheta, \nabla_x \vartheta) = - \kappa(\vr,\vartheta) \nabla_x \vartheta,
	\end{align}
	where $ \kappa $ is the \emph{heat conductivity} coefficient.
	Since we consider the bounded domain $ \Om $, our goal is to discuss the solvability of the initial-boundary value problem for the system \eqref{NSF-c}-\eqref{NSF-e} endowed with the constitutive relations \eqref{visc_str} and \eqref{Fou_law} where we consider \emph{homogeneous Dirichlet boundary condition} in velocity and \emph{in-homogeneous Dirichlet boundary condition} in the temperature
	\begin{align}
		&\vc{u}|_{\partial \Om} = 0, \label{velocity_bdry}\\
		&\vartheta|_{\partial \Om} = \vartheta_B>0. \label{temp_bdry}
	\end{align}
 \par 
  The consideration of \eqref{temp_bdry} is closely related to the celebrated \emph{Rayleigh-Benard problem}.
  For the above system \eqref{NSF-c}-\eqref{NSF-e} with the consitutive relations \eqref{GE}-\eqref{temp_bdry} and the pressure following Boyle's law, the existence of a\emph{ local in time} strong solution was proved by Valli and Zajaczkowski \cite{VZ1986}. For the global in time weak solutions there are many articles and monographs by P.L. lions,\cite{PLL1996} Bresch and Desardin\cite{BD2007}, Feireisl and Novotn\'y \cite{FN2009b}, Bresch and Jabin \cite{BJ2018} which focus mainly on space-periodic domains or energy-conserving boundary conditions for the temperature. Also for thermodynamically open systems, Feireisl and Novotn\'y \cite{FN2021} prove the existence of weak solutions, but also require the control of the internal (heat) energy flux $ \vc{q} $ in $ \partial \Om $. Recently, the concept of weak solution was introduced for the particular boundary condition \eqref{temp_bdry} in \cite{CF2021} and it also considers very general in flow-outflow boundary data. This weak solution approach is slightly different from the earlier considerations that we will discuss in Section 2. This similar idea is adapted by Pokorn\' y in \cite{P2022} for steady flows.\par
	The concept of measure valued solutions in the context of inviscid and incompressible fluids has been used by several authors, beginning with the seminal work of DiPerna \cite{D1985} and DiPerna and Majda \cite{DM1987}, see also related results of Kr\"oner and Zajaczkowski \cite{KJ1996}, Necasova and Novotn\'y\cite{NN1994} and Neustupa \cite{Nes1993}. B\v rezina, Feireisl and Novotn\'y\cite{BFN2020} also introduce the measure valued solution for compressible viscous heat conducting fluids for the "no flux" boundary condition. Their definition follows the approach of considering measure valued solution as a general object as described by Brenier, De Lellis and Szekelyhidi \cite{BDS11}, in contrast to the approach described by Malek et. al. \cite{MNRR1996}, where they identify a measure valued solution as a weak limit of weak solutions. We follow here the approach proposed by \cite{BFN2020}.\par

For a generalized (weak or measure valued) solution, it is quite important to prove the \emph{generalized-strong uniqueness principle}. The principle asserts that given the same initial data a weak solution will coincide with the strong or classical solution if the latter exists. For such properties in the context of weak solutions, see Feireisl and Novotn\'y\cite{FN2012} and  Feireisl \cite{F2012}. For measure valued solution for the system was avalaible in B\v rezina et al\cite{BFN2020}. The main idea here is to use the \emph{relative energy inequality}. It was first introduced by Daefermos \cite{D1979} for scalar conservation law and the for compressible fluid a suitable adaptation is available in Feireisl and Novotn\'y \cite{FN2009b}. \par

An important application of measure valued solutions is their identification as limits of numerical schemes. Interesting results in numerical analysis have been obtained by Fjordholm et al.  \cite{FjMT2012}, Feireisl and  Luk\'a\v cov\'a-Medvid'ov\'a \cite{FL2018} and Feireisl, Luk\'a\v cov\'a-Medvid'ov\'a and Mizerov\'a\cite{FLM2020i}. Together with the existing weak(measure valued)-strong uniqueness principle in the class of measure valued solutions, one can show that numerical solutions converge strongly to a strong solution of the system as long as the latter exists. In the context of compressible Navier-Stokes-Fourier system with no-flux boundary condition see Feireisl et al\cite{FLMS2021} and a detailed discussion is available in the monograph by Feireisl et al. \cite{FLMS2021book}. \par
 The plan for the article is as follows:
	\begin{itemize}
		\item At first in Section \ref{s2}, we devote ourselves to the weak formulation of the problem along with a proper definition of the measure valued solution.
		\item In Section \ref{s3}, adapting the relative energy inequality suitably for the measure valued solution, we derive the relative energy inequality for the measure valued solution.
		\item We present our main results in Section \ref{s4}. We conclude the weak(measure valued)-strong uniqueness property for the system. Our first two theorems, Theorem \ref{thc1} and Theorem \ref{thc2}, requires some additional hypothesis on measure valued solutions, while Theorem \ref{thu} is does not need any of such extra assumption on solutions, although we impose some physically relevant structural assumptions on transport coefficients $ \mu, \lambda $ and $ \kappa $.
		\item Finally, in the Section \ref{s5}, we discuss briefly the limitation of our results, some comments on existence and validation of the definition.
	\end{itemize}

	\section{Measure valued solution}\label{s2}
	\subsection{Weak formulation: Revisit}
		At first we quickly recall the weak formulation described in Feireisl and Novotn\'y \cite[Chapter 2]{FN2009b}. Although it deals with a no-flux boundary condition for temperature, i.e. 
		\begin{align}\label{temp-bdry-2}
			\vc{q} \cdot \vc{n}=0 \text{ on } \partial \Om.
		\end{align}
	The weak formulation of the problem depends on the weak formulation of continuity equation, momentum equation and replaces the \emph{internal energy equation} \eqref{NSF-e} by the \emph{entropy inequality} 
	\begin{align}\label{ent_ineq}
		\begin{split}
			\partial_t (\vr s(\vr, \vt)) + \Div (\vr s(\vr, \vt) \vu) + \Div \left( \frac{ \vc{q}(\vr,\vt, \nabla_x \vt ) }{\vt} \right)\br \geq \frac{1}{\vt} 
		\left( \mathbb{S} (\vr,\vt, \bb{D}_x \vu ) : \bb{D}_x \vu - \frac{\vc{q}(\vr,\vt, \nabla_x \vt) \cdot \nabla_x \vt }{\vt} \right),
	\end{split}
	\end{align}
	along with the weak form of \emph{total energy balance} : 
		\begin{align}  
		\frac{\D }{\dt} &\intO{ \left[ \frac{1}{2} \vr |\vu |^2 + \vr e \right] } 		=		\intO{ \vr \vu \cdot \vc{g}  } .
	\end{align}
	Instead of boundary condition \eqref{temp-bdry-2}, if we consider inhomogeneous Dirichlet boundary condition for the temperature \eqref{temp_bdry}, we need to proceed as prescribed \cite[Section 2.4]{CF2021}. In this case, assuming all quantities in consideration are smooth, we have 
	\begin{align}
		\partial_t \left[  \frac{1}{2} \vr |\vu |^2 + \vr e  \right] + \Div\left( \left[  \frac{1}{2} \vr |\vu |^2 + \vr e +p(\vr,\vt)\right] \vu  \right) + \Div \vc{q}=\Div\left(\bb{S}\vu\right)+\vr \vu \cdot \vc{g},
	\end{align}
and consequently
	\begin{align}  
		\frac{\D }{\dt} &\intO{ \left[ \frac{1}{2} \vr |\vu |^2 + \vr e \right] } + 
		\int_{\partial \Omega} \vc{q} \cdot \vc{n} \ \D \sigma_x  
		=		\intO{ \vr \vu \cdot \vc{g}  } 
		\label{w1}
	\end{align}
where $ \sigma_x $ is the surface measure on the boundary $ \partial \Om $. Therefore, the total energy balance is unavailable. But, considering a smooth function $ \tilde{\vt} $ such that \[ \tvt >0 \text{ in } (0,T) \times \Om \text{ with } \tvt = \vt_B \text{ on }\partial \Om,\]
and multiplying the entropy inequality \eqref{ent_ineq} by $\tvt$ along with performing the integrating by parts formula, we obtain 
\begin{align} 
	- \frac{\D }{\dt} &\intO{ \tvt \vr s } - \int_{\partial \Omega} \vc{q} \cdot \vc{n}\, \D \sigma_x 
	 \nonumber\\ &\leq - \intO{ \frac{\tvt}{\vt}	 \left( \mathbb{S} : \bb{D}_x \vu - \frac{\vc{q} \cdot \nabla \vt }{\vt} \right) }
	- \intO{ \left[ \vr s \left( \partial_t \tvt + \vu \cdot \nabla \tvt \right) + \frac{\vc{q}}{\vt} \cdot \nabla \tvt \right] }.
	\label{w2}
\end{align}
Adding \eqref{w1} and \eqref{w2} we obtain the \emph{ballistic energy inequality}
	\begin{align*}  
	\frac{\D }{\dt} &\intO{ \left[ \frac{1}{2} \vr |\vu |^2 + \vr e - \tvt \vr s \right] } + \intO{ \frac{\tvt}{\vt}	 \left( \mathbb{S} : \bb{D}_x \vu - \frac{\vc{q} \cdot \nabla \vt }{\vt} \right) }\br
	&\leq 	\intO{ \vr \vu \cdot \vc{g}  }-  \intO{ \left[ \vr s \left( \partial_t \tvt + \vu \cdot \nabla \tvt \right) + \frac{\vc{q}}{\vt} \cdot \nabla \tvt \right] }.
\end{align*}
We recall that for some smooth function $ \tT $, the ballistic energy is denoted by $H_{\tT}(\vr,\vt)   $  and it reads as \[ H_{\tT}(\vr,\vt) = \vr [e(\vr,\vt)- \tT s(\vr,\vt)] .\] 
Therefore we consider the weak formulation with ballistic energy inequality instead of energy balance. 
In this paper our goal is to provide a suitable definition of a measure valued solution based on the above discussion for boundary condition \eqref{temp_bdry}. To define it, we take motivation from B\v rezina, Feireisl and Novotn\' y \cite{BFN2020} that covers the the boundary condition \eqref{temp-bdry-2}. This definition will give in terms of Young measure and suitable defect measures. For a simpler consideration, from now on we consider $ \vc{g}=0 $.
	\subsubsection{Phase space and Young measure} A natural candidate for the phase space is given by the state variables $ [\varrho,\vc{u},\vartheta] $. Since we are looking for a more general class of solutions and $ \nabla_x \vc{u} $  and $ \nabla_x \vartheta $ are present in the system \eqref{NSF-c}-\eqref{NSF-e}, thus gradient of velocity and temperature have been included along with the natural choices. Hence a proper phase space is 
	\begin{align}
		\mathcal{F}=\{[\varrho,\vc{u},\vartheta,\DD_{\vectoru}, \vc{D}_{\vartheta}] \big | \varrho\in [0,\infty),\; \vectoru\in \R^d,\; \vartheta \in [0,\infty),\; \DD_{\vectoru}\in \R^{d\times d}_{\text{sym}},\; \vc{D}_\vartheta \in \R^{d} \}.
	\end{align}

	We consider a Young measure $ \Nu $ such that $\Nu \equiv \{ \Nu_{t,x} \}_{(t,x)\in (0,T)\times \Omega}$ and
	\begin{align}\label{mv-ym}
		\Nu \in L^{\infty}_{\text{weak-(*)}} \big( (0,T)\times \Om ;\mathcal{P}(\mathcal{F} )\big),
	\end{align}
	\subsubsection{Initial data}
	
	Let $ \Nu_0 \in L^{\infty}_{\text{weak-(*)}} \big( \Om ;\mathcal{P}(\mathcal{F}_0 )\big) $ such that 
	$ \mathcal{F}_0= \{ (\varrho_0, \vc{u}_0, \vt_0)| \vr_0 \in [0,\infty),\, \vc{u}_0\in \R^d,\, \vt_0\in [0,\infty)\} $ with finite energy, i.e., 
	
	\begin{align}\label{mv-ic}
		\int_{ \Om} \left\langle \Nu_{0,x}; \left( \frac{1}{2}\varrho \vert \mathbf{u} \vert^2 + \varrho e(\varrho,\vartheta) - \tT(0,x) \varrho s(\varrho,\vartheta)\right) \right\rangle \dx < \infty
	\end{align}
	for each $0< \tT  \in C^1([0,T] \times \Ov{\Om}) $ with $ \tT= \vt_B $ on $\partial \Om $.
	\subsubsection{Compatibility relations for the Young measure} 
	We consider a very general phase space for the Young measure. It satisfies the following compatibility conditions.
	\begin{itemize}
		\item \textbf{Velocity compatibility:} The identity
		\begin{align}\label{mv-vc}
			-\int_{0}^{T} \int_{ \Om}\lel \Nu_{t,x}; \vc{u} \ril \cdot \Div \bb{T} \dx \dt = \int_{0}^{T} \int_{ \Om} \lel \Nu_{t,x}; \bb{D}_{\vc{u}} \ril \dx \dt 
		\end{align}
		holds for any $ \bb{T} \in C^1(\Ov{Q_T}; \R^{d\times d}_{\text{sym}}) $. 
		\item \textbf{Temperature compatibility:} The identity
		\begin{align}\label{mv-tc}
			-\int_{0}^{T} \int_{ \Om} \lel  \Nu_{t,x}; \vt - \widetilde{\Theta} \ril  \Div \psi\dx \dt =  -\int_{0}^{T} \int_{ \Om} \lel \Nu_{t,x} ; \vc{D}_\vt - \nabla_x \widetilde{\Theta} \ril \psi \dx \dt 
		\end{align}
		holds for any $ \psi \in C^1([0,T]\times \Ov{\Om}) $ and $ \widetilde{\Theta} \in C^1([0,T]\times \Ov{\Om})  $ with $ \widetilde{\Theta}= \vt_B $ on $\partial \Om $.
	\end{itemize}
	\subsubsection{Field equations with defect}
	\begin{itemize}
		\item \textbf{Equation of continuity:} For a.e. $\tau \in (0,T) $ and $\psi \in C^{1}([0,T]\times \bar{\Om})$, the following equation holds:
		\begin{align} \label{mv-cont-eqn}
			\begin{split}
				&\int_{ \Om} \langle \Nu_{\tau,x} ; \varrho \rangle \psi(\tau, \cdot) \dx - \int_{ \Om} \langle \Nu_{0,x} ; \varrho \rangle \psi(0, \cdot) \dx \\
				&\quad = \int_{0}^{\tau} \int_{ \Om} \big[ \langle \Nu_{t,x}; \varrho  \rangle \partial_{t}\psi + \langle \Nu_{t,x}; \varrho \vc{u} \rangle \cdot \nabla_x \psi \big] \dx \dt.
			\end{split}
		\end{align}
		\item \textbf{Momentum equation:} There exists a \textit{measure} $r^{M} \in L^\infty_{\text{weak-(*)}}(0,T;\mathcal{M}(\bar{\Om};\R^{d\times d}))$ such that for a.e. $\tau \in (0,T)$ and every $\vectorphi\in C^{1}([0,T]\times \bar{\Om};\R^d)$, $\vectorphi|_{\partial \Om} =0$ the following equation holds:
		\begin{align}\label{mv-mom-eqn}
			\begin{split}
				&\int_{ \Om} \langle \Nu_{\tau,x}; \varrho \vc{u} \rangle \cdot \vectorphi(\tau,\cdot) \dx - \int_{ \Om} \langle \Nu_{0,x} ;\varrho \vc{u} \rangle \cdot \vectorphi(0,\cdot) \dx\\
				&= \int_{0}^{\tau} \int_{ \Om} \big[ \langle \Nu_{t,x}; \varrho \vc{u} \rangle \cdot \partial_{t} \vectorphi +  \langle \Nu_{t,x}; \varrho \vc{u} \otimes \mathbf{u}) \rangle : \nabla_x \vectorphi + \langle \Nu_{t,x}; p(\varrho,\vartheta) \rangle \Div \vectorphi  \big] \dx \dt\\
				& - \int_{0}^{\tau} \int_{ \Om} \langle \Nu_{t,x};\mathbb{S}(\varrho,\vartheta,\DD_{\vectorv}) \rangle : \nabla_x \vectorphi \dx \dt + \int_{0}^{\tau} \langle r^M;\nabla_x \vectorphi \rangle_{\{\mathcal{M}(\bar{\Om};\R^{d \times d}),C(\bar{\Om};\R^{d \times d})\}} \dt.
			\end{split}
		\end{align}
		\item \textbf{Entropy inequality:} For a.e. $\tau \in (0,T) $ and $0\leq \phi \in C_c^{1}([0,T]\times {\Om})$, we have the following inequality:
		\begin{align}\label{mv-ent-ineq}
			\begin{split}
				&\int_{ \Om} \langle \Nu_{\tau,x} ; \varrho s(\varrho,\vartheta) \rangle \phi(\tau,x) \dx -	\int_{ \Om} \langle \Nu_{0,x} ; \varrho s(\varrho,\vartheta)  \rangle \phi(0,x) \dx \\
				&\geq \int_0^\tau \int_{ \Om} \langle \Nu_{t,x} ; \varrho s(\varrho,\vartheta)  \rangle \partial_t \phi(t,x) \dx \dt \\
				& + \int_0^\tau \int_{ \Om} \left\langle \Nu_{t,x} ; \varrho s(\varrho,\vartheta)  \vc{u} - \frac{\kappa(\vr,\vartheta)}{\vartheta} \vc{D}_\vartheta \right\rangle \cdot \nabla_x \phi(t,x) \dx \dt \\
				&+ \int_0^\tau \int_{ \Om} \left\langle \Nu_{t,x} ; \frac{1}{\vartheta} \left( \bb{S}(\varrho,\vartheta,\nabla_{x}\vc{u}) \colon \bb{D}_x \vc{u} + \frac{\kappa(\vr,\vartheta)}{\vartheta} \vert \vc{D}_\vartheta \vert^2 \right) \right\rangle   \phi(t,x) \dx \dt 
			\end{split}
		\end{align}
		\item \textbf{Ballistic energy inequality:} For any $\tT \in C^1([0,T]\times \Ov{\Om}),\, \tT>0,\; \tT|_{\partial\Om}= \vartheta_B$, there exists a \emph{dissipation defect} $\mathcal{D}_{\tT}$ such that
		\begin{align*}
			\mathcal{D}_{\tT}\in L^{\infty}(0,T),\; \mathcal{D}_{\tT}\geq 0,
		\end{align*} and the following inequality holds:
		\begin{align}\label{Bal-eng-ineq}
			\begin{split}
				&\int_{ \Om} \left \langle\Nutx ; \left( \frac{1}{2}\varrho \vert \mathbf{u} \vert^2 + \varrho e(\varrho,\vartheta) - \tT \varrho s(\varrho,\vartheta)\right) \right \rangle \dx \\
				&+ \int_0^\tau \int_{ \Om} \left\langle \Nu_{t,x} ; \frac{1}{\vartheta} \left( \bb{S}(\varrho,\vartheta,\nabla_{x}\vc{u}) \colon \bb{D}_x \vc{u} + \frac{\kappa(\vr,\vartheta)}{\vartheta} \vert \vc{D}_\vartheta \vert^2 \right) \right\rangle   \tT(t,x) \dx \dt + \mathcal{D}_{\tT}(\tau) \\
				&\leq  \int_{ \Om} \left\langle \Nu_{0,x}; \left( \frac{1}{2}\varrho \vert \mathbf{u} \vert^2 + \varrho e(\varrho,\vartheta) - \tT \varrho s(\varrho,\vartheta)\right) \right\rangle \dx \\ 
				&- \int_0^\tau \int_{ \Om} \langle\Nu_{t,x} \varrho s(\varrho,\vartheta) \rangle \partial_{t} \tT\dx \dt -  \int_0^\tau \int_{ \Om} \langle\Nu_{t,x} \varrho s(\varrho,\vartheta) \vc{u} \rangle \cdot \nabla_x \tT\dx \dt  \\
				&+\int_0^\tau \int_{ \Om} \left\langle \Nu_{t,x}; \frac{\kappa(\vr,\vt) }{\vartheta}\mathbf{D}_\vartheta\right\rangle \cdot \nabla_x \tT \dx \dt.
			\end{split}
		\end{align}
	\end{itemize}
	\subsubsection{Compatibility of defect measures}
	For any $\tT \in C^1([0,T]\times \Ov{\Om}),\, \tT>0,\; \tT|_{\partial\Om}= \vartheta_B$, we have
	\begin{equation}\label{def-m-comp}
		\vert \langle r^M(\tau);\nabla_x \vectorphi \rangle_{\{\mathcal{M}(\bar{\Om};\R^{d \times d}),C(\bar{\Om};\R^{d\times d})\}} \vert \leq \xi(\tau) \mathcal{D}_{\tT}(\tau) \Vert \vectorphi \Vert_{C^1(\bar{\Om})}
	\end{equation} 
	where $ \xi \in L^1(0,T) $.
	\subsubsection{Generalized Korn--Poincar\'e inequality}
	The following version of Korn-Poincar\' e inequality is true: 
	\begin{align}\label{mv-KP}
		\inttom \lel \Nu_{t,x}; \vert \vc{u} - \widetilde{\vc{U}} \vert^2 \ril \dx \dt \leq C_p \inttom \lel \Nu_{t,x}; \vert \bb{D}_0(\bb{D}_\vc{u}) - \bb{D}(\nabla_x \widetilde{\vc{U}} ) \vert^2 \ril \dx \dt
	\end{align}
	for any $ \widetilde{\vc{U}}  \in L^2(0,T; W^{1,2}_0(\Om; \R^d)) $.
	\begin{Rem}
	We must note that the ballistic energy inequality is given for a large class of functions $ \tT $ and in general the dissipation defect $ \mathcal{D}_{\tT} $ is dependent of $ \tilde{\Theta} $. On the other hand the defect measure $ (r^M) $ in \eqref{mv-mom-eqn} is independent of $ \tT $. Therefore, the \eqref{def-m-comp} is very strong assumption. Although for some physical equation state we will able to conclude that $ \mathcal{D}_{\tT} $ is independent of $ \tT $, we will discuss it in Section \ref{s5}.
	\end{Rem}
	\begin{Rem}
		Analogously, on can think of to have a defect measure in the right hand side of inequality \eqref{Bal-eng-ineq}. We are avoiding it. An explanation is available in \ref{s5}.
	\end{Rem}
	\subsection{Definition of a measure valued solution}
   Here now we provide the definition of the measure valued solution 
   \begin{Def}\label{def:m}
   	Let $ \vt_B \in C^1((0,T)\times \partial \Om) $ with $ \vt_B >0 $ and $ \tT $ belongs to the class $ \{ \tT \in C^1([0,T]\times \Ov{\Om})\,| \tT>0,\; \tT|_{\partial\Om}= \vartheta_B\} $. Moreover, the initial condition $ \Nu_0 $ satisfies \eqref{mv-ic}. Then $ \{\Nu,\mathcal{D}_{\tT}\} $ is a measure valued solution for the system \eqref{NSF-c}-\eqref{NSF-e} with \eqref{GE}, \eqref{visc_str}, \eqref{Fou_law}, \eqref{velocity_bdry} and \eqref{temp_bdry} if it satisfies \eqref{mv-cont-eqn}-\eqref{Bal-eng-ineq} with the compatibility condition \eqref{mv-vc},\eqref{mv-tc}, \eqref{def-m-comp} and \eqref{mv-KP}. 
   \end{Def}

	\section{Relative energy inequality }\label{s3}
	
Following Feireisl and Novotn\' y in \cite[Chapter 9]{FN2009b}, we consider the relative energy with the help of Ballistic energy and it is given by  
	\begin{align*}
		E(\vrthu\,\vert\, \tildevrthu) = \frac{1}{2} \vr \vert \vc{u} -\tu \vert^2 + \left( H_{\tvt} (\vr,\vt) - \frac{\partial H_{\tvt} (\tvr,\tvt) }{\partial \vr } (\vr -\tvr) - H_{\tvt} (\tvr,\tvt) \right),
	\end{align*}
	where $ (\tildevrthu) $ are smooth functions such that they satisfy
	\begin{align}
		\tvr>0 , \tu \text{ with }\tu\vert_{\partial \Om}=0 \text{ and } \tvt>0 \text{ with } \tvt\vert_{\Om} = \vt_B .
	\end{align}
As observed in \cite[Section 1.2]{FN2021}, if the relative energy is interpreted in terms of the conservative entropy variables $(\vr, S = \vr s, \vm = \vr \vu)$, represents a Bregman distance/divergence associated with the energy functional 
\[
E(\vr, S, \vm) = \frac{1}{2} \frac{|\vm|^2}{\vr} + \vr e(\vr, S).
\]
To obtain that the reltive enrgy functional is non-negative, Indeed we need the \emph{hypothesis of thermodynamic stability}, i.e.
\begin{align}\label{ther-stab}
	\frac{\partial p(\vr, \vt)}{\partial \vr} > 0,\; \frac{\partial e(\vr, \vt)}{\partial \vt} > 0,
\end{align} 
which in turn yields the convexity of the internal energy $\vr e(\vr,S)$ with respect to the variables $(\vr,S)$. In addition,
\begin{equation} \label{ws2}
	\frac{\partial (\vr e(\vr, S))}{\partial \vr} = e - \vt s + \frac{p}{\vr},\ 
	\frac{\partial (\vr e(\vr, S))}{\partial S} = \vt.
\end{equation} 
Thus the relative energy expressed in the conservative entropy variable may be interpreted as 
\[
E \left( \vr, S, \vm \Big| \tvr , \tilde{S}, \tilde{\vm} \right) = E(\vr, S, \vm) - \left< \partial E(\tvr, \tilde{S}, \tilde{\vm}) ; (\vr - \tvr, S - \tilde{S}, \vm - \tilde{\vm}) \right> - E(\tvr, \tilde{S}, \tilde{\vm}).
\]
	The time evolution of the relative energy is given by \[\Epsilon(\relenttuple) = \int_{ \Om} E(\relenttuple) \dx .\] 
	Given a measure valued solution $ \{\Nu_{t,x}\}_{(t,x)\in Q_T} $ of Navier-Stokes-Fourier system, we adapt the relative energy as 
	\begin{align}\label{rel1}
		\Epsilon_{mv}(\tau):= \int_{ \Om} \lel \Nu_{\tau,x} ; E(\relenttuple) \ril \dx .
	\end{align}
	At first, using the standard expansion, we have 
	\begin{align}\label{rel2}
		\begin{split}
			\Epsilon_{mv}(\tau)= &\int_{ \Om} \lel \Nu_{\tau,x} ; \frac{1}{2} \vr \vert \vc{u} \vert^2 + H_{\tvt}(\vr,\vt)  \ril \dx- \int_{ \Om} \lel \Nu_{\tau,x} ; \varrho \vc{u} \ril  \cdot \tu \dx  \\
			& + \int_{ \Om} \lel \Nu_{\tau,x} ; \varrho \ril \left( \frac{1}{2} \vert \tu \vert^2 - \frac{\partial H_{\tvt} (\tvr,\tvt) }{\partial \vr} \right) \dx + \int_{ \Om} p(\tvr,\tvt)(\tau,\cdot) \dx = \Sigma_{i=1}^{4} \mathcal{L}_i.
		\end{split}
	\end{align}
	In the identity \eqref{rel2}, the term $ \mathcal{L}_1 $ is associated with the ballistic energy inequality \eqref{Bal-eng-ineq}, while for the terms $ \mathcal{L}_2 $ and $ \mathcal{L}_3 $ we use the momentum equation \eqref{mv-mom-eqn} and the continuity equation \eqref{mv-cont-eqn}, respectively. For the term $ \mathcal{L}_4 $ we use the identity 
	\[ \int_0^ \tau \int_{ \Om} \partial_{t} p(\tvr,\tvt) \dx \dt = \left[  \int_{ \Om} p(\tvr,\tvt)(t,\cdot) \dx\right]_{t=0}^{t=\tau} .\] 
	A suitable application of the Gibbs' relation \eqref{GE} yields
	\begin{align}\label{rel3}
		\begin{split}
			\Epsilon_{mv}(\tau) & +  \int_0^ \tau \int_{ \Om} \left[ \lel \Nu_{t,x} ; \frac{\tvt(t,x)}{\vt}\bb{S}(\vr,\vt, \bb{D}_{\vc{u}}) \colon \bb{D}_{\vc{u}} \ril + \tvt(t,x) \lel \Nu_{t,x} ; \frac{\kappa(\vr,\vt)}{\vt^2} \vert \vc{D}_\vt \vert^2\ril \right] \!\!\dx \dt\\
			& - \inttom \lel \Nu_{t,x} \colon \bb{S}(\vr,\vt, \bb{D}_{\vc{u}}) \ril \colon \nabla_x \tu(t,x) \dx \dt + \mathcal{D}(\tau)\\
			\leq &\Epsilon_{mv}(0)-\inttom \bigg[ \lel \Nu_{t,x} ; \vr s(\vr,\vt) \ril \partial_{t} \tvt (t,x) \\ 
			&      \hspace{ 22mm } +\lel \Nutx ; \left( \vr s(\vr,\vt)\vc{u} - \frac{\kappa(\vr,\vt)}{\vt} \vc{D}_\vt \right) \ril \cdot \nabla_x \tvt (t,x) \bigg] \dx \dt	\\
			& + \inttom \big[ \lel \Nutx ; \vr ( \tu(t,x)-\vc{u}) \ril \cdot \partial_{t} \tu(t,x) \\ 
			&      \hspace{ 22mm } + \lel \Nu_{t,x} ; \vr( \tu(t,x)-\vc{u}) \otimes \vc{u} \ril \colon \nabla_x \tu(t,x)   \big] \dx \dt\\
			& -\inttom \lel \Nu_{t,x} ; p(\varrho,\vt ) \ril \Div \tu(t,x) \dx \dt \\
			& + \inttom \left[ \lel \Nu_{t,x} ; \vr \ril \partial_{t} \tvt(t,x) s(\tvr,\tvt)(t,x)  + \lel \Nu_{t,x} ; \vr \vc{u} \ril \cdot  \nabla_x \tvt(t,x) \, s(\tvr,\tvt)(t,x) \right]\\
			& + \inttom \left[ \lel \Nu_{t,x} ; \tvr(t,x)- \vr \ril \frac{1}{\tvr} \partial_{t} p(\tvr,\tvt)(t,x)  - \lel \Nu_{t,x} ; \vr \vc{u} \ril \cdot \frac{1}{\tvr}  \nabla_x \partial_{t} p(\tvr,\tvt)(t,x)  \right]\\
			& +  \int_{0}^{\tau} \langle r^M;\nabla_x \tu(t,x) \rangle_{\{\mathcal{M}(\bar{\Om};\R^{d \times d}),C(\bar{\Om};\R^{d \times d})\}} \dt.
		\end{split}
	\end{align}
	Our main goal is to establish the weak (measure valued)-- strong uniqueness property. To obtain this result we choose $ (\tvr,\tu,\tvt) $ as a strong solution of the problem emanating from the same initial data $ \Nu_{0,x} $ and they share same boundary condition. Thus, suppose that $ (\tvr,\tu,\tvt) $ is smooth and $ \vr, \vt >0 $ in $ (0,T)\times \Om $, then the inequality\eqref{rel3} reduces to 
	\begin{align*}
			\Epsilon_{mv}(\tau) & + \int_0^ \tau \int_{ \Om} \left[ \lel \Nu_{t,x} ; \frac{\tvt(t,x)}{\vt}\bb{S}(\vr,\vt, \bb{D}_{\vc{u}}) \colon \bb{D}_{\vc{u}} \ril \right] \dx\dt\\
			& + \inttom \tvt(t,x) \lel \Nu_{t,x} ; \frac{\kappa(\vr,\vt)}{\vt}  \vc{D}_\vt \left( \frac{\vc{D}_\vt}{\vt} - \frac{\nabla_x \tvt}{\tvt}\right)\ril  \!\!\dx \dt\\
			& - \inttom \lel \Nu_{t,x} \colon \bb{S}(\vr,\vt, \bb{D}_{\vc{u}}) \ril \colon \nabla_x \tu(t,x) \dx \dt \\
			& - \inttom \lel \Nu_{t,x}; (\bb{D}_\vc{u}- \bb{D}(\tu)) \ril \colon \bb{S}(\tvr,\tvt,\nabla_x \tu) \dx \dt + \mathcal{D}(\tau)\\
			\leq &\Epsilon_{mv}(0) - \inttom \tvr \lel \Nutx ; s(\vr,\vt)-s(\tvr,\tvt) \ril \left( \partial_{t} \tvt + \tu \cdot \nabla_x \tvt \right) \dx \dt \\
			& - \inttom \lel \Nu_{t,x} ; \vc{u}- \tu \ril \cdot \nabla_x p(\tvr,\tvt) \dx \dt -\inttom \lel \Nu_{t,x} ; p(\varrho,\vt ) \ril \Div \tu(t,x) \dx \dt \\ 
			& + \inttom \left[ \lel \Nu_{t,x} ; \tvr(t,x)- \vr \ril \frac{1}{\tvr} \partial_{t} p(\tvr,\tvt)(t,x)  - \lel \Nu_{t,x} ; \vr \vc{u} \ril \cdot \frac{1}{\tvr}  \nabla_x  p(\tvr,\tvt)(t,x)  \right]\\
			& +  \int_{0}^{\tau} \langle r^M;\nabla_x \tu(t,x) \rangle_{\{\mathcal{M}(\bar{\Om};\R^{d \times d}),C(\bar{\Om};\R^{d \times d})\}} \dt + \inttom R_1 \dx \dt,
	\end{align*} 
	where
	\begin{align*}
			R_1= &\lel \Nu_{t,x}; \vr (\tu(t,x) - \vc{u}) \otimes (\vc{u}- \tu(t,x)) \ril \colon \bb{D}_x \tu(t,x) \\
			&+ \lel \Nu_{t,x}; \left( \frac{\vr}{\tvr(t,x)}-1\right)(\tu(t,x) -\vc{u}) \ril \cdot \Div \bb{S}(\tvt, \bb{D}_x \tu)(t,x)\\
			& +\lel \Nu_{t,x}; \left( \frac{\vr}{\tvr(t,x)}-1\right)(\tu(t,x) -\vc{u})  \ril  \cdot \nabla_x p(\tvr,\tvt)(t,x) \\
			& - \lel \Nu_{t,x}; \vr (s(\vr,\vt)- s(\tvr,\tvt)) (\tu-\vc{u})\ril \cdot \nabla_x \tvt \\
			&- \lel \Nu_{t,x}; (\vr-\tvr)(s(\vr,\vt)-s(\tvr,\tvt)) \ril \left( \partial_{t} \tvt + \tu \cdot \nabla_x \tvt \right).
	\end{align*}
Moreover, adjusting a few terms suitably, the above inequality becomes
	\begin{align*}
			\Epsilon_{mv}(\tau) & + \int_0^ \tau \int_{ \Om} \left[ \lel \Nu_{t,x} ; \frac{\tvt(t,x)}{\vt}\bb{S}(\vr,\vt, \bb{D}_{\vc{u}}) \colon \bb{D}_{\vc{u}} \ril \right] \dx\dt\\
			& + \inttom \tvt(t,x) \lel \Nu_{t,x} ; \frac{\kappa(\vr,\vt)}{\vt}  \bb{D}_\vt \left( \frac{\vc{D}_\vt}{\vt} - \frac{\nabla_x \tvt}{\tvt}\right)\ril  \dx \dt\\
			& - \inttom \lel \Nu_{t,x} \colon \bb{S}(\vr,\vt, \bb{D}_{\vc{u}}) \ril \colon \nabla_x \tu(t,x) \dx \dt \\
			& - \inttom \lel \Nu_{t,x}; (\bb{D}_\vc{u}- \bb{D}(\nabla_x \tu)) \ril \colon \bb{S}(\tvr,\tvt,\nabla_x \tu) \dx \dt + \mathcal{D}_{\tvt}(\tau)\\
			\leq &\Epsilon_{mv}(0) + \inttom \lel \Nutx ; \tvt-\vt \ril \Div\left( \frac{\kappa(\tvr,\tvt) \nabla_x \tvt}{\tvt}\right)\dx \dt \\
			& +\inttom \lel \Nutx ; \left(1 - \frac{\vt}{\tvt}\right) \ril \left( \bb{S}(\tvr,\tvt,\bb{D}(\nabla_x \tu))\colon \bb{D}(\nabla_x) \tu + \frac{\kappa(\tvr,\tvt) \nabla_x \tvt}{\tvt}\right) \dx\dt \\
			& +  \int_{0}^{\tau} \langle r^M;\nabla_x \tu(t,x) \rangle_{\{\mathcal{M}(\bar{\Om};\R^{d \times d}),C(\bar{\Om};\R^{d \times d})\}} \dt + \inttom R_2 \dx \dt,
	\end{align*} 
	where 
	\begin{align*}
		R_2=&R_1+ \lel \Nu_{t,x}; \left( 1- \frac{\vr}{\tvr(t,x)}\right)(\tu(t,x) -\vc{u})  \ril  \cdot \nabla_x p(\tvr,\tvt)(t,x) \\
		&+ \lel \Nu_{t,x} ; p(\tvr,\tvt) - \frac{\partial p(\tvr,\tvt)}{\partial \vr} (\tvr-\vr)- \frac{\partial p(\tvr,\tvt)}{\partial \vt} (\tvt -\vt) - p(\vr,\vt) \ril \Div \tu \\
		& + \lel \Nu_{t,x} ; s(\tvr,\tvt) - \frac{\partial s(\tvr,\tvt)}{\partial \vr} (\tvr-\vr)- \frac{\partial s(\tvr,\tvt)}{\partial \vt} (\tvt -\vt) - s(\vr,\vt) \ril \tvr \left( \partial_{t} \tvt + \tu \cdot \nabla_x \tvt \right).
	\end{align*}
			More precisely, we have
			\begin{align}\label{R2}
				\begin{split}R_2=&\lel \Nu_{t,x}; \vr (\tu(t,x) - \vc{u}) \otimes (\vc{u}- \tu(t,x)) \ril \colon \bb{D}_x \tu(t,x) \\
			&+ \lel \Nu_{t,x}; \left( \frac{\vr}{\tvr(t,x)}-1\right)(\tu(t,x) -\vc{u}) \ril \cdot \left(\Div \bb{S}(\tvt, \bb{D}_x \tu)(t,x)+ \nabla_x p(\tvr,\tvt)(t,x)\right)\\
			& - \lel \Nu_{t,x}; \vr (s(\vr,\vt)- s(\tvr,\tvt)) (\tu-\vc{u})\ril \cdot \nabla_x \tvt \\
			&- \lel \Nu_{t,x}; (\vr-\tvr)(s(\vr,\vt)-s(\tvr,\tvt)) \ril \left( \partial_{t} \tvt + \tu \cdot \nabla_x \tvt \right)\\
			 &+\lel \Nu_{t,x}; \left( 1- \frac{\vr}{\tvr(t,x)}\right)(\tu(t,x) -\vc{u})  \ril  \cdot \nabla_x p(\tvr,\tvt)(t,x) \\
			&+ \lel \Nu_{t,x} ; p(\tvr,\tvt) - \frac{\partial p(\tvr,\tvt)}{\partial \vr} (\tvr-\vr)- \frac{\partial p(\tvr,\tvt)}{\partial \vt} (\tvt -\vt) - p(\vr,\vt) \ril \Div \tu \\
			& + \lel \Nu_{t,x} ; s(\tvr,\tvt) - \frac{\partial s(\tvr,\tvt)}{\partial \vr} (\tvr-\vr)- \frac{\partial s(\tvr,\tvt)}{\partial \vt} (\tvt -\vt) - s(\vr,\vt) \ril \tvr \left( \partial_{t} \tvt + \tu \cdot \nabla_x \tvt \right) 
		\end{split}
	\end{align}
Now we use the compatibility of the Young measure \eqref{mv-vc} and \eqref{mv-tc} to deduce
	\begin{align}\label{REineq2}
		\begin{split}
			\Epsilon_{mv}(\tau) & + \int_0^ \tau \int_{ \Om} \left[ \lel \Nu_{t,x} ; \frac{\tvt}{\vt}\bb{S}(\vr,\vt, \bb{D}_{\vc{u}}) \colon \bb{D}_{\vc{u}} \ril + \lel \Nu_{t,x}; \frac{\vt}{\tvt} \ril \bb{S}(\tvr.\tvt, \nabla_x \tu)\right] \dx\dt\\
			& - \inttom  \lel \Nu_{t,x} \colon \bb{S}(\vr,\vt, \bb{D}_{\vc{u}}) \ril \colon \nabla_x \tu \dx\dt -\inttom \lel \Nu_{t,x}; \bb{D}_\vc{u} \ril \colon \bb{S}(\tvr,\tvt,\nabla_x \tu) \dx \dt\\
			& + \inttom \tvt \lel \Nu_{t,x} ; \frac{\kappa(\vr,\vt)}{\vt}  \vc{D}_\vt \left( \frac{\vc{D}_\vt}{\vt} - \frac{\nabla_x \tvt}{\tvt}\right)\ril  \!\!\dx \dt\\
			& + \inttom \kappa(\tvr,\tvt) \frac{\nabla_x \tvt}{\tvt} \cdot \lel \Nu_{t,x}; \vt \left( \frac{\nabla_x \vt}{\vt}- \frac{\vc{D}_\vt}{\vt} \right) \ril \dx \dt + \mathcal{D}_{\tvt}(\tau)\\
			\leq &\Epsilon_{mv}(0) +  \int_{0}^{\tau} \langle r^M;\nabla_x \tu(t,x) \rangle_{\{\mathcal{M}(\bar{\Om};\R^{d \times d}),C(\bar{\Om};\R^{d \times d})\}} \dt + \inttom R_2 \dx \dt,
		\end{split}
	\end{align} 
	where the term $ R_2 $ is given by \eqref{R2} and it consists of quadratic error terms.
	Let us now fix some notation to write \eqref{REineq2} more precisely. For $ \bb{A}\left(=(a_{ij})_{i,j=1}^{d} \right)\in \R^{d\times d} $, we consider the \textit{symmetric part} and the \textit{traceless part} of $ \bb{A} $ as
	\begin{align*}
		\bb{D}(\bb{A})= \frac{\bb{A}+\bb{A}^T}{2}\text{ and }
		\bb{D}_0(\bb{A}) =\frac{\bb{A}+\bb{A}^T}{2}-\frac{1}{d}\text{Tr}(\bb{A})\, \bb{I},
	\end{align*} 
	respectively, where $ \text{Tr}(\bb{A})  =\sum\limits_{i=1}^d a_{ii}$. 
	We write the Newtonian stress tensor as
	\begin{align*}
		\mathbb{S}(\nabla_x \mathbf{u})&=\mu(\varrho , \vartheta) \bigg(\frac{\nabla_x \vectoru + \nabla_x^{T} \vectoru}{2}-\frac{1}{d} (\Div\vectoru)\mathbb{I} \bigg) + \lambda(\varrho , \vartheta) (\Div \vectoru) \mathbb{I},\\
		&= \mu(\varrho , \vartheta) \bb{D}_0 \vc{u} + \lambda(\varrho,\vt) \Div \vc{u} \bb{I}.
	\end{align*}
	
	We rewrite the inequality \eqref{REineq2} as 
	\begin{align}\label{REineq3}
		\begin{split}
			\Epsilon_{mv}(\tau) &+ \inttom \lel \Nu_{t,x}; \mu(\vr,\vt)  \frac{\tvt}{\vt} \left\vert \bb{D}_0 (\du) - \frac{\vt}{\tvt} \bb{D}_0 (\nabla_x \tu )  \right\vert^2 \ril \dx\dt \\
			&+ \inttom \bb{D}_0(\nabla_x \tu) \colon \lel  \Nu_{t,x}; (\mu(\vr,\vt)- \mu(\tvr,\tvt) )\left( \bb{D}_0 (\du) - \frac{\vt}{\tvt} \bb{D}_0 (\nabla_x \tu )  \right)  \ril\dx\dt\\
			& +\inttom \lel \Nu_{t,x}; \lambda(\vr,\vt)  \frac{\tvt}{\vt} \left\vert \text{Tr} (\du) - \frac{\vt}{\tvt} \Div \tu  \right\vert^2 \ril \dx\dt \\
			&+ \inttom \Div \tu  \lel  \Nu_{t,x}; (\lambda(\vr,\vt)- \lambda(\tvr,\tvt) )\left( \text{Tr} (\du) - \frac{\vt}{\tvt} (\Div \tu )  \right)  \ril\dx\dt\\
			& + \inttom \tvt \lel \Nu_{t,x} ; \kappa(\vr,\vt)  \left\vert \frac{\vc{D}_\vt}{\vt} - \frac{\nabla_x \tvt}{\tvt}\right\vert^2 \ril  \!\!\dx \dt\\
			& + \inttom \kappa(\tvr,\tvt) \frac{\nabla_x \tvt}{\tvt} \cdot \lel \Nu_{t,x}; (\vt-\tvt) \left( \frac{\nabla_x \tvt}{\tvt}- \frac{\vc{D}_{\vt}}{\vt} \right) \ril \dx \dt \\
			&- \inttom {\nabla_x \tvt} \cdot \lel \Nu_{t,x}; (\kappa(\vr,\vt)-\kappa(\tvr,\tvt) \left( \frac{\nabla_x \tvt}{\tvt}- \frac{\vc{D}_{\vt}}{\vt} \right) \ril \dx \dt + \mathcal{D}_{\tvt}(\tau)\\
			\leq &\Epsilon_{mv}(0) +  \int_{0}^{\tau} \langle r^M;\nabla_x \tu \rangle_{\{\mathcal{M}(\bar{\Om};\R^{d \times d}),C(\bar{\Om};\R^{d \times d})\}} \dt + \inttom R_2 \dx \dt.
		\end{split}
	\end{align} 
The inequality \eqref{REineq3} is called the \emph{relative energy inequality} associated with the problem. Therefore we summarize the above discussion in the following lemma:
\begin{Lemma}\label{lem-1}
Let the transport coefficients $ \kappa(\varrho,\vartheta) $, $ \mu(\varrho,\vt)  $ and $ \lambda(\vr,\vt) $ be continuously differentiable and positive for $ \varrho >0,\; \vt >0  $. Let the thermodynamic functions satisfy Gibbs equation\eqref{GE} and the thermodynamic stability assumption\eqref{ther-stab}. Let $ \{\Nu, \mathcal{D}_{\tvt}\}$ be a measure valued solution of the system \eqref{NSF-c}-\eqref{NSF-e} with initial data $ \Nu_0 $ and $ \{\tvr, \tu,\tvt\} $ be a strong solution with sufficient regularity and initial data. Then we have the inequality \eqref{REineq3} holds, where the remainder term $ R_2 $ is given by \eqref{R2}.
	
\end{Lemma}
\subsection{A suitable reduction of relative energy inequality}
In this subsection we will try to reduce the \eqref{REineq3}. Since we already notice that the relative energy is a non-negative functional and the remainder term $ R_2 $ contains certain quadratic terms. Therefore we introduce the next part to have a close look on relative energy.
\subsubsection{Essential and residual part of a function}\label{ess-res-ss}
	At first, we introduce a cut-off function
	$\chi_\delta $ such that
	\begin{align}\nonumber
		\begin{split}
			\chi_\delta \in C_c^{\infty}	\left((0, \infty)^2\right),\; 0 \leq \chi_\delta \leq 1,\; \chi_\delta(\varrho,\vt) = 1  \mbox{if}\ \delta \leq \varrho \leq \frac{1}{\delta} \text{ and }\delta \leq \vt \leq \frac{1}{\delta} \text{ for some }\delta>0.
		\end{split}
	\end{align}
	For a function $H = H(\vr,\vt,\vc{u}, \du, \dvt)$, we set
	\begin{align}\nonumber
		\begin{split}
			[H]_{\text{ess}}=\chi_\delta(\varrho,\vt) H(\vr,\vt,\vc{u}, \du, \dvt) ,\; [H]_{\text{res}}=(1- \chi_\delta(\varrho,\vt)) H(\vr,\vt,\vc{u}, \du, \dvt).
		\end{split}
	\end{align}	
	If $ \tvr, \tvt $ is strictly positive, bounded above and bounded below, $ \vert \tu \vert $ is also bounded, then with the help of the above notation, we have
	\begin{align}\label{RE_ess_res}
		\begin{split}
		E (\vrthu\mid \tildevrthu)
		\geq c(\delta,\tvr,\tvt,\tu) \bigg( & \left[ \levert \vr-\tvr \rivert^2+\levert \vt-\tvt \rivert^2+\levert \vc{u}-\tu \rivert^2 \right]_{\text{ess}} \\
		& + \left[ 1 +\varrho+ \vr \vert s(\vr, \vt) \vert+ \vr e(\vr,\vt) + \vr \vert \vc{u} \vert^2 \right]_{\text{res}} \bigg).
		\end{split}
	\end{align}
\begin{Rem}
	The inequality \eqref{RE_ess_res} implies that the relative energy functional is coercive.
\end{Rem}
At this point, our goal is either to control the remainder term $ R_2$ of \eqref{REineq3}, by the integral of the relative energy or to absorb it into a non-negative term on the left-hand side of \eqref{REineq3}. 
\begin{Lemma}
	Let the hypothesis in Lemma \ref{lem-1} remains true. Then with the help of \eqref{RE_ess_res}, the inequality \eqref{REineq3} reduces to
	\begin{align}\label{REineq4}
		\begin{split}
			\Epsilon_{mv}(\tau) &+ \inttom \lel \Nu_{t,x}; \mu(\vr,\vt)  \frac{\tvt}{\vt} \left\vert \bb{D}_0 (\du) - \frac{\vt}{\tvt} \bb{D}_0 (\nabla_x \tu )  \right\vert^2 \ril \dx\dt \\
			&+ \inttom \bb{D}_0(\nabla_x \tu) \colon \lel  \Nu_{t,x}; (\mu(\vr,\vt)- \mu(\tvr,\tvt) )\left( \bb{D}_0 (\du) - \frac{\vt}{\tvt} \bb{D}_0 (\nabla_x \tu )  \right)  \ril\dx\dt\\
			& +\inttom \lel \Nu_{t,x}; \lambda(\vr,\vt)  \frac{\tvt}{\vt} \left\vert \text{Tr} (\du) - \frac{\vt}{\tvt} \Div \tu  \right\vert^2 \ril \dx\dt \\
			&+ \inttom \Div \tu  \lel  \Nu_{t,x}; (\lambda(\vr,\vt)- \lambda(\tvr,\tvt) )\left( \text{Tr} (\du) - \frac{\vt}{\tvt} (\Div \tu )  \right)  \ril\dx\dt\\
			& + \inttom \tvt \lel \Nu_{t,x} ; \kappa(\vr,\vt)  \left\vert \frac{\vc{D}_\vt}{\vt} - \frac{\nabla_x \tvt}{\tvt}\right\vert^2 \ril  \!\!\dx \dt\\
			& + \inttom \kappa(\tvr,\tvt) \frac{\nabla_x \tvt}{\tvt} \cdot \lel \Nu_{t,x}; (\vt-\tvt) \left( \frac{\nabla_x \tvt}{\tvt}- \frac{\vc{D}_{\vt}}{\vt} \right) \ril \dx \dt \\
			&- \inttom {\nabla_x \tvt} \cdot \lel \Nu_{t,x}; (\kappa(\vr,\vt)-\kappa(\tvr,\tvt) \left( \frac{\nabla_x \tvt}{\tvt}- \frac{\vc{D}_{\vt}}{\vt} \right) \ril \dx \dt + \mathcal{D}_{\tvt}(\tau)\\
			\leq &\Epsilon_{mv}(0) + C(\delta,\tildevrthu)  \int_{0}^\tau \Epsilon_{mv}(t) \dt + \int_{0}^\tau \mathcal{D}_{\tvt}(t) \dt \\
			& + \int_{0}^\tau \int_{ \Om} \lel  \Nu_{t,x};  \left[\vt +\vert p(\vr,\vt) \vert + \vert \vc{u}-\tu \vert +  \vr \vert s(\vr, \vt) \vert\; \vert \vc{u} \vert  \right]_{\text{res}} \ril \dx\dt .
		\end{split}
	\end{align} 
\end{Lemma}
\begin{proof}
The proof is quite straight forward, since $ R_2  $ contains quadratic terms, therefore the essential parts of $ R_2 $ will be controlled by $ C(\delta,\tildevrthu)  \int_{0}^\tau \Epsilon_{mv}(t) \dt  $. For the residual parts of $ R_2 $, we control some of the terms that are compatible with the right hand side of \eqref{RE_ess_res} and we keep the other terms we in the right hand side of the \eqref{REineq4}.
\end{proof}
	\section{Main results: Weak (measure valued)-strong uniqueness }\label{s4}
	\subsection{Conditional Weak(measure valued)--strong uniqueness}
	By the term `conditional', we mean that, along with the measure valued solution if we assume some additional hypothesis, them we will able to achieve the desired  \emph{Weak(measure valued)--strong uniqueness} property.
	\subsubsection{First conditional result}
	First conditional result when density and temperature are uniformly bounded.
	\begin{Th}\label{thc1}
		Let the transport coefficients $ \kappa(\varrho,\vartheta) $, $ \mu(\varrho,\vt)  $ and $ \lambda(\vr,\vt) $ be continuously differentiable and positive for $ \varrho >0,\; \vt >0  $. Let the thermodynamic functions satisfy Gibbs equation \eqref{GE} and the thermodynamic stability \eqref{ther-stab}. Further assume that $ (\tilde{\vr}, \tilde{\vc{u}}, \tilde{\vt} )$ be a strong solution of the system in $ [0,T]\times \Om  $ emanating from the initial data $ (\vr_0, \vc{u}_0, \vt_0 )$ with $ \vr_0>0, \vt_0 $.  Assume that $ \Nu $ be a measure valued solution of the same problem following the Definition \ref{def:m} and it satisfies  		\begin{align}\label{hyp_con1}
			\Nu_{t,x} \left\{0< \underline{\vr} < \varrho < \Ov{\vr},\; 0<\underline{\vt}< \vt < \overline{\vt}  \right\} =1 {\text{ for a.e. }} (t,x)\in Q_T,
		\end{align}
		for some constants $ \underline{\vr}, \Ov{\vr},\Un{\vt},\Ov{\vt} >0 $ and initial data coincides, i.e.
		\begin{align*}
			\Nu_{0,x}= \delta_{[\vr_0,\vc{u}_0,\vt_0]} \text{ for a.e. } x\in \Om. 
		\end{align*}
		Then \begin{align*}
			\Nu_{t,x}= \delta_{[ \tvr(t,x), \tu(t,x), \tvt(t,x), \bb{D}_x \tu(t,x), \nabla_x \tvt(t,x)]}\text{ for a.e. } (t,x) \in Q_T.
		\end{align*}
	\end{Th}
\begin{proof}
	
	At first we note that there exists a $ \delta >0 $, such that $ \delta \leq \min\{ \Un{\vr}, \Un{\vt}\} $ and $ \max{\{\Ov{\vr},\Ov{\vt}\}} \leq \frac{1}{\delta}$.  There for we consider a cut-off function $ \chi_\delta $ as mentioned in \ref{ess-res-ss}. The hypothesis \eqref{REineq4} implies \begin{align*}
		\int_{0}^\tau \int_{ \Om} \lel  \Nu_{t,x};  \left[\vt +\vert p(\vr,\vt) \vert + \vert \vc{u}-\tu \vert +  \vr \vert s(\vr, \vt) \vert\; \vert \vc{u} \vert  \right]_{\text{res}} \ril \dx\dt =0.
	\end{align*} 
Also, form the assumption \eqref{hyp_con1}, we observe 

	\begin{align*}
		& \inttom \lel \Nu_{t,x}; \mu(\vr,\vt)  \frac{\tvt}{\vt} \left\vert \bb{D}_0 (\du) - \frac{\vt}{\tvt} \bb{D}_0 (\nabla_x \tu )  \right\vert^2 \ril \dx\dt \\
		&+\inttom \lel \Nu_{t,x}; \lambda(\vr,\vt)  \frac{\tvt}{\vt} \left\vert \text{Tr} (\du) - \frac{\vt}{\tvt} \Div \tu  \right\vert^2 \ril \dx\dt \\
		& + \inttom \tvt \lel \Nu_{t,x} ; \kappa(\vr,\vt)  \left\vert \frac{\vc{D}_\vt}{\vt} - \frac{\nabla_x \tvt}{\tvt}\right\vert^2 \ril  \!\!\dx \dt\\
		&\geq \underline{C}(\delta, \tildevrthu)  \inttom \bigg\langle \Nu_{t,x};  \left\vert \bb{D}_0 (\du) - \frac{\vt}{\tvt} \bb{D}_0 (\nabla_x \tu )  \right\vert^2 \\
		&\hspace{42mm}+\left\vert \text{Tr} (\du) - \frac{\vt}{\tvt} \Div \tu  \right\vert^2  
		+ \left\vert \frac{\vc{D}_\vt}{\vt} - \frac{\nabla_x \tvt}{\tvt}\right\vert^2  \bigg\rangle \dx\dt \\
	\end{align*}
and 
\begin{align*}
	&+ \inttom \bb{D}_0(\nabla_x \tu) \colon \lel  \Nu_{t,x}; (\mu(\vr,\vt)- \mu(\tvr,\tvt) )\left( \bb{D}_0 (\du) - \frac{\vt}{\tvt} \bb{D}_0 (\nabla_x \tu )  \right)  \ril\dx\dt\\
	&+ \inttom \Div \tu  \lel  \Nu_{t,x}; (\lambda(\vr,\vt)- \lambda(\tvr,\tvt) )\left( \text{Tr} (\du) - \frac{\vt}{\tvt} (\Div \tu )  \right)  \ril\dx\dt\\
	&- \inttom {\nabla_x \tvt} \cdot \lel \Nu_{t,x}; (\kappa(\vr,\vt)-\kappa(\tvr,\tvt) \left( \frac{\nabla_x \tvt}{\tvt}- \frac{\vc{D}_{\vt}}{\vt} \right) \ril \dx \dt \\
	& + \inttom \kappa(\tvr,\tvt) \frac{\nabla_x \tvt}{\tvt} \cdot \lel \Nu_{t,x}; (\vt-\tvt) \left( \frac{\nabla_x \tvt}{\tvt}- \frac{\vc{D}_{\vt}}{\vt} \right) \ril \dx \dt \\
	&\leq C(\epsilon, \tildevrthu)  \inttom \lel \Nu_{t,x}; \levert \vr-\tvr \rivert^2+\levert \vt-\tvt \rivert^2 \ril \dx \dt \\
	&+ \epsilon  \inttom \bigg\langle \Nu_{t,x};  \left\vert \bb{D}_0 (\du) - \frac{\vt}{\tvt} \bb{D}_0 (\nabla_x \tu )  \right\vert^2 \\
	&\hspace{42mm}+\left\vert \text{Tr} (\du) - \frac{\vt}{\tvt} \Div \tu  \right\vert^2  
	+ \left\vert \frac{\vc{D}_\vt}{\vt} - \frac{\nabla_x \tvt}{\tvt}\right\vert^2  \bigg\rangle \dx\dt ,
\end{align*}
where $ \underline{C} (\delta,\tildevrthu)>0 $.
Therefore, choosing $ \epsilon $ suitably, we deduce
	\begin{align}\label{ws1,1}
	\begin{split}
		\Epsilon_{mv}(\tau) &+ \frac{1}{2}\underline{C} (\delta,\tildevrthu)  \inttom \bigg\langle \Nu_{t,x};  \left\vert \bb{D}_0 (\du) - \frac{\vt}{\tvt} \bb{D}_0 (\nabla_x \tu )  \right\vert^2 \\
		&\hspace{42mm}+\left\vert \text{Tr} (\du) - \frac{\vt}{\tvt} \Div \tu  \right\vert^2  
		+ \left\vert \frac{\vc{D}_\vt}{\vt} - \frac{\nabla_x \tvt}{\tvt}\right\vert^2  \bigg\rangle \dx\dt + \mathcal{D}_{\tvt}(\tau)\\
		\leq &\Epsilon_{mv}(0) + C(\delta,\tildevrthu)  \int_{0}^\tau \Epsilon_{mv}(t) \dt + \int_{0}^\tau \mathcal{D}_{\tvt}(t) \dt .
	\end{split}
\end{align} 
 Employing Gr\" onwall's lemma we obtain the desired result.
\end{proof}
 	\subsubsection{Second conditional result}
 Our next result is for a more physical situation with where pressure law is boyle's law. But the weak-strong uniqueness we prove is imposing certain condition on entropy and additional hypothesis on transport coefficients. 
 \begin{itemize}
 	\item The bulk and shear viscosity coefficient are given by 
 	\begin{align}\label{tc1}
 		\mu(\vr,\vt)= C_\mu (1+\vt) \text{ and } 	\lambda(\vr,\vt)= C_\lambda (1+\vt) \text{ with } C_\mu>0,\; \C_\lambda \geq 0. 
 	\end{align}
 \item The heat conductivity coefficient is given by 
 \begin{align}\label{tc2}
 	\kappa(\vr,\vt) = \kappa(1+\vt) \text{ with } \kappa >0.
 \end{align}
 \end{itemize}

\begin{Th}\label{thc2}
		Let the transport coefficients $ \kappa(\varrho,\vartheta) $, $ \mu(\varrho,\vt)  $ and $ \lambda(\vr,\vt) $ be given by \eqref{tc1} and \eqref{tc2}. Let the thermodynamic functions satisfy Gibbs equation and the thermodynamic stability assumption with \begin{align}\label{pr-law-per}
			p(\vr,\vt)=\vr \vt,\; 	e(\vr,\vt)=c_v \vt ,\; 	s(\vr,\vt)=\log \left(\frac{\vt^{c_v}}{\vr}\right),\; c_v>1.
		\end{align}
	Further assume that $ (\tilde{\vr}, \tilde{\vc{u}}, \tilde{\vt}) $ be a strong solution of the system in $ [0,T]\times \Om  $ emanating from the initial data $ (\vr_0, \vc{u}_0, \vt_0) $ with $ \vr_0>0, \vt_0 $.  Assume that $ \Nu $ be a measure valued solution of the same problem following the Definition \ref{def-m-comp} and it satisfies  		\begin{align}\label{hyp_con2}
		\Nu_{t,x} \left\{ \vert s(\vr,\vt) \vert \leq \bar{s}  \right\} =1 {\text{ for a.e. }} (t,x)\in Q_T,
	\end{align}
	for some $ \bar{s}>0 $ and 
	\begin{align*}
		\Nu_{0,x}= \delta_{[\vr_0,\vc{u}_0,\vt_0]} \text{ for a.e. } x\in \Om. 
	\end{align*}
	Then \begin{align*}
		\Nu_{t,x}= \delta_{[ \tvr(t,x), \tu(t,x), \tvt(t,x), \bb{D}_x \tu(t,x), \nabla_x \tvt(t,x)]}\text{ for a.e. } (t,x) \in Q_T.
	\end{align*}
\end{Th}
\begin{proof}
	
For the sake of simplicity we assume $ C_\lambda=0 $. Using the hypothesis \eqref{tc1}, we get
\begin{align*}
	& \inttom \lel \Nu_{t,x}; \mu(\vr,\vt)  \frac{\tvt}{\vt} \left\vert \bb{D}_0 (\du) - \frac{\vt}{\tvt} \bb{D}_0 (\nabla_x \tu )  \right\vert^2 \ril \dx\dt \\
	& \geq  \inttom \lel \Nu_{t,x}; \frac{1}{2}C_\mu (1+\vt)  \frac{\tvt}{\vt} \left\vert \bb{D}_0 (\du) - \frac{\vt}{\tvt} \bb{D}_0 (\nabla_x \tu )  \right\vert^2 \ril \dx\dt \\
	& +   \inttom \lel \Nu_{t,x}; \frac{1}{4} \tvt \left\vert \bb{D}_0 (\du) -  \bb{D}_0 (\nabla_x \tu )  \right\vert^2 \ril \dx\dt \\
	& - \inttom  \frac{1}{4}\lel \Nu_{t,x}; \tvt  \left\vert \left(1-\frac{\vt}{\tvt}\right) \bb{D}_0 (\nabla_x \tu )  \right\vert^2 \ril \dx\dt 
\end{align*}
and
\begin{align*}
	&\levert \inttom \bb{D}_0(\nabla_x \tu) \colon \lel  \Nu_{t,x}; (\mu(\vr,\vt)- \mu(\tvr,\tvt) )\left( \bb{D}_0 (\du) - \frac{\vt}{\tvt} \bb{D}_0 (\nabla_x \tu )  \right)  \ril\dx\dt \rivert \\
	& \leq \epsilon  \inttom \lel \Nu_{t,x};  \left\vert \bb{D}_0 (\du) - \frac{\vt}{\tvt} \bb{D}_0 (\nabla_x \tu )  \right\vert^2 \ril \dx\dt \\
	&+  C(\epsilon, \tvt,\tu ) \inttom \lel \Nu_{t,x}; \levert \vt-\tvt \rivert^2 \ril \dx \dt. 
\end{align*}
Also, \eqref{tc2} implies
\begin{align*}
	& \inttom \tvt \lel \Nu_{t,x} ; \kappa(1+\vt)  \left\vert \frac{\vc{D}_\vt}{\vt} - \frac{\nabla_x \tvt}{\tvt}\right\vert^2 \ril  \!\!\dx \dt \geq C(\tvt) \inttom \tvt \lel \Nu_{t,x} ;   \left\vert \frac{\vc{D}_\vt}{\vt} - \frac{\nabla_x \tvt}{\tvt}\right\vert^2 \ril  \!\!\dx \dt.
\end{align*}
Similarly, using Young's inequality, we infer the following estimates
\begin{align*}
	 &\levert \inttom \kappa(\tvt) \frac{\nabla_x \tvt}{\tvt} \cdot \lel \Nu_{t,x}; (\vt-\tvt) \left( \frac{\nabla_x \tvt}{\tvt}- \frac{\vc{D}_{\vt}}{\vt} \right) \ril \dx \dt \rivert \\
	 &\leq \epsilon \inttom  \lel \Nu_{t,x} ;   \left\vert \frac{\vc{D}_\vt}{\vt} - \frac{\nabla_x \tvt}{\tvt}\right\vert^2 \ril  \!\!\dx \dt +  C(\epsilon, \tvt ) \inttom \lel \Nu_{t,x}; \levert \vt-\tvt \rivert^2 \ril \dx \dt
\end{align*}
and 
\begin{align*}
	&\levert \inttom {\nabla_x \tvt} \cdot \lel \Nu_{t,x}; (\kappa(\vr,\vt)-\kappa(\tvr,\tvt) \left( \frac{\nabla_x \tvt}{\tvt}- \frac{\vc{D}_{\vt}}{\vt} \right) \ril \dx \dt\rivert\\
	&\leq \epsilon \inttom  \lel \Nu_{t,x} ;   \left\vert \frac{\vc{D}_\vt}{\vt} - \frac{\nabla_x \tvt}{\tvt}\right\vert^2 \ril  \!\!\dx \dt +  C(\epsilon, \tvt ) \inttom \lel \Nu_{t,x}; \levert \vt-\tvt \rivert^2 \ril \dx \dt.
\end{align*}
Now, for pressure law \eqref{pr-law-per}, we note that 
$$ \vert p(\vr,\vt) \vert \leq 1+ \vr e(\vr,\vt) + \vr \vert s(\vr,\vt) \vert .$$ Then using compatibility condition \eqref{mv-KP} and \eqref{RE_ess_res} we obtain
	\begin{align}\label{ws2}
	\begin{split}
		\Epsilon_{mv}(\tau) &+ \frac{1}{2}\inttom \lel \Nu_{t,x}; C_\mu(1+\vr,\vt)  \frac{\tvt}{\vt} \left\vert \bb{D}_0 (\du) - \frac{\vt}{\tvt} \bb{D}_0 (\nabla_x \tu )  \right\vert^2 \ril \dx\dt \\
		&+   \inttom \lel \Nu_{t,x}; \frac{1}{4} \tvt \left\vert \bb{D}_0 (\du) -  \bb{D}_0 (\nabla_x \tu )  \right\vert^2 \ril \dx\dt \\
		& + \inttom \tvt \lel \Nu_{t,x} ; \kappa(1+\vt)  \left\vert \frac{\vc{D}_\vt}{\vt} - \frac{\nabla_x \tvt}{\tvt}\right\vert^2 \ril  \!\!\dx \dt + \mathcal{D}_{\tvt}(\tau)\\
		\leq &\Epsilon_{mv}(0) + C(\delta)  \int_{0}^\tau \Epsilon_{mv}(t) \dt + \int_{0}^\tau \mathcal{D}_{\tvt}(t) \dt \\
		& + \int_{0}^\tau \int_{ \Om} \lel  \Nu_{t,x};  \left[\vt^2 + \vr \vert s(\vr, \vt) \vert\; \vert \vc{u} \vert  \right]_{\text{res}} \ril \dx\dt   +\epsilon \inttom  \lel \Nu_{t,x} ;   \left\vert \frac{\vc{D}_\vt}{\vt} - \frac{\nabla_x \tvt}{\tvt}\right\vert^2 \ril  \!\!\dx \dt \\
		&+  \epsilon  \inttom \lel \Nu_{t,x};  \left\vert \bb{D}_0 (\du) - \frac{\vt}{\tvt} \bb{D}_0 (\nabla_x \tu )  \right\vert^2 \ril \dx\dt \\
		&+  \epsilon  \inttom \lel \Nu_{t,x};  \left\vert \bb{D}_0 (\du) - (\nabla_x \tu )  \right\vert^2 \ril \dx\dt .
	\end{split}
\end{align} 

Now using \eqref{hyp_con2} we have $ 	\vt^{c_v} \leq C(\bar{s}) \vr$ and $ \vert \vr s \vu \vert \leq C(\bar{s} )\left (\vr + \vr \vert \vu \vert^2\right)$.
This implies 
\[ \vt^{c_v+1} \leq \vr \vt = \frac{1}{c_v} \vr e (\vr,\vt). \]
By proper choice of $ \epsilon $, inequality \eqref{ws2} reduces to 
\begin{align*}
		\Epsilon_{mv}(\tau) &+ \frac{1}{4}\inttom \lel \Nu_{t,x}; C_\mu(1+\vr,\vt)  \frac{\tvt}{\vt} \left\vert \bb{D}_0 (\du) - \frac{\vt}{\tvt} \bb{D}_0 (\nabla_x \tu )  \right\vert^2 \ril \dx\dt \\
		&+  \frac{1}{8} \inttom \lel \Nu_{t,x};  \tvt \left\vert \bb{D}_0 (\du) -  \bb{D}_0 (\nabla_x \tu )  \right\vert^2 \ril \dx\dt \\
		& +\frac{1}{2} \inttom \tvt \lel \Nu_{t,x} ; \kappa(1+\vt)  \left\vert \frac{\vc{D}_\vt}{\vt} - \frac{\nabla_x \tvt}{\tvt}\right\vert^2 \ril  \!\!\dx \dt\\
		&+ \mathcal{D}_{\tvt}(\tau)\\
		\leq &\Epsilon_{mv}(0) + C(\delta)  \int_{0}^\tau \Epsilon_{mv}(t) \dt + \int_{0}^\tau \mathcal{D}_{\tvt}(t) \dt .
\end{align*} 
Again, invoking Gr\" onwall's lemma, we complete the proof.
\end{proof}
	\subsection{Unconditional weak(measure valued)--strong uniqueness}
	Here we consider a general pressure law and some structural assumption on the transport coefficients. We call it unconditional result because, here we will not assume any further information on measure valued solution.
	Let us first consider the pressure law in the following way:
	\begin{align}\label{Pr1}
		p(\vr,\vt)= p_M(\vr,\vt)+ p_R(\vr,\vt),
	\end{align}
	where $ p_M $ stands for the \emph{molecular pressure }and $ p_R  $ is the \emph{radiation pressure}. The relation between the molecular pressure $ p_M $ and the associated internal energy $ e_M $ is is given by
	\begin{align}\label{pm,em}
		p_M(\vr,\vt)= \frac{2}{3} \varrho e_M(\vr,\vt). 
	\end{align}
 From the Gibb's relation \eqref{GE}, we have 
	\[p_M(\vr,\vt) = \vt^{\frac{5}{2}} P\left( \frac{\vr}{\vt^{\frac{3}{2}}}\right)  ,\] 
	for some function $ P $. Moreover, following \cite{FN2009b}, we assume that $ P\in C^1[0,\infty) \cap C^5(0,\infty)$,
	\begin{align}
		\begin{split}
			& P(0)=0, \; P^\prime(q)>0,\text{ for all } q>0,\\
			&0< \frac{\frac{5}{3} P(q)- P^\prime(q)q }{q} <c \text{ for all }q < 0, \lim\limits_{q\rightarrow \infty } \frac{P(q)}{q^{\frac{5}{3}}}= \overline{p}>0. 
		\end{split}
	\end{align}
	We rewrite the internal energy $ (e_M) $ associated with molecular pressure as \[ e_M(\vr)= \frac{3}{2} \frac{\vt^{\frac{5}{2}}}{\varrho} P\left( \frac{\vr}{\vt^{\frac{3}{2}}}\right), \] 
	and, again using Gibbs relation, we have 
	\begin{align}\label{sm}
		s_M= S\left( \frac{\vr}{\vt^{\frac{3}{2}}}\right),
	\end{align}
	for some function $ S $ with the property 
	\begin{align}
		S^\prime (q)= -\frac{3}{2}  \frac{\frac{5}{3} P(q)- P^\prime(q)q }{q^2} <0 . 
	\end{align}
	Finally, we impose the third law of thermodynamics in the form 
	\begin{align}\label{3rd law}
		\lim\limits_{q\rightarrow \infty } S(q) =0. 
	\end{align} 

the structural assumptions on transport coefficient reads as
\begin{align}\label{k,mu,lamda}
	\begin{split}
	&\kappa(\vr,\vt)= \kappa_1 + \kappa_2 \vt^\beta, \text{with} \kappa_1>0,\; \kappa_2 \geq 0 \text{ and } \beta \leq 2 \\
	&  \mu(\varrho,\vt) = \mu_0+ \mu_1 \vt,\; \mu_0,\mu_1\geq 0,\\
	 &\lambda(\varrho,\vt) = \lambda_0+ \lambda_1 \vt,\; \lambda_0,\lambda_1\geq 0.
\end{split}
\end{align}
	\begin{Th}\label{thu}
		 Let the pressure follows \eqref{Pr1} with $ p_M, e_M  $ and $ s_M $ given by \eqref{pm,em} and \eqref{sm}, respectively, $ \kappa,\mu, \lambda $ satisfy \eqref{k,mu,lamda} and the radiation pressure along with associated internal energy and entropy are given by \begin{align}\label{rad_p1}
			p_R= a \vt^2,\; e_R= a \frac{\vt^2}{\varrho} \text{ and } s_R= 2a\frac{\vt}{\vr} \text{ with } a>0.
		\end{align}
		Assume that $ [\tvr,\tu,\tvt] $ is a classical solution to the system in $ [0,T] \times \Om  $ emanating from initial data $ [\vr_0,\vc{u}_0,\vt_0] $ with $  \vr_0,\vt_0>0 $ in $ \Ov{\Om} $. Assume further that $ \{ \Nu_{t,x}\}_{(t,x)\in Q_T} $ is a measure valued solution of the same problem following the Definition \ref{def-m-comp} such that \[ \Nu_{0,x}= \delta_{[\vr_0,\vc{u}_0,\vt_0]} \text{ for a.e. } x \in \Om .\]
		Then \begin{align}
			\Nu_{t,x}= \delta_{[ \tvr(t,x), \tu(t,x), \tvt(t,x), \bb{D}_x \tu(t,x), \nabla_x \tvt(t,x)]}\text{ for a.e. } (t,x) \in Q_T.
		\end{align}
	\end{Th}
\begin{proof}
	
At first, using the structural assumption \eqref{k,mu,lamda}, we have
\begin{align*}
	& \inttom \lel \Nu_{t,x}; \mu(\vr,\vt)  \frac{\tvt}{\vt} \left\vert \bb{D}_0 (\du) - \frac{\vt}{\tvt} \bb{D}_0 (\nabla_x \tu )  \right\vert^2 \ril \dx\dt \\
	&+ \inttom \bb{D}_0(\nabla_x \tu) \colon \lel  \Nu_{t,x}; (\mu(\vr,\vt)- \mu(\tvr,\tvt) )\left( \bb{D}_0 (\du) - \frac{\vt}{\tvt} \bb{D}_0 (\nabla_x \tu )  \right)  \ril\dx\dt\\
	&=  \inttom \lel \Nu_{t,x}; \mu_0  \frac{\tvt}{\vt} \left\vert \bb{D}_0 (\du) - \frac{\vt}{\tvt} \bb{D}_0 (\nabla_x \tu )  \right\vert^2 \ril \dx\dt \\
	& + \inttom \lel \Nu_{t,x}; \mu_1  \tvt \left\vert \bb{D}_0 (\du) - \frac{\vt}{\tvt} \bb{D}_0 (\nabla_x \tu )  \right\vert^2 \ril \dx\dt \\
	&+\inttom \bb{D}_0(\nabla_x \tu) \colon \lel  \Nu_{t,x}; \mu_1(\vt -\tvt   )\left( \bb{D}_0 (\du) - \frac{\vt}{\tvt} \bb{D}_0 (\nabla_x \tu )  \right)  \ril\dx\dt\\
	&= \inttom \lel \Nu_{t,x}; \mu_0  \frac{\tvt}{\vt} \left\vert \bb{D}_0 (\du) - \frac{\vt}{\tvt} \bb{D}_0 (\nabla_x \tu )  \right\vert^2 \ril \dx\dt \\
	&+ \inttom \lel \Nu_{t,x}; \mu_1 \tvt \left\vert \bb{D}_0 (\du) - \bb{D}_0 (\nabla_x \tu )  \right\vert^2 \ril \dx\dt \\
	 &-\inttom \bb{D}_0(\nabla_x \tu) \colon \lel  \Nu_{t,x}; \mu_1(\vt -\tvt   )\left( \bb{D}_0 (\du) - \bb{D}_0 (\nabla_x \tu )  \right)  \ril\dx\dt \\
	&+ \inttom \mu_1 \bb{D}_0(\nabla_x \tu) \colon \lel  \Nu_{t,x}; \frac{1}{\tvt}\mu_1(\vt -\tvt   )^2  \bb{D}_0 (\nabla_x \tu )    \ril\dx\dt.
\end{align*}
Using Young's inequality, we obtain
\begin{align*}
	 &\inttom \bb{D}_0(\nabla_x \tu) \colon \lel  \Nu_{t,x}; \mu_1(\vt -\tvt   )\left( \bb{D}_0 (\du) - \bb{D}_0 (\nabla_x \tu )  \right)  \ril\dx\dt \\
	 &\leq \epsilon  \inttom \lel \Nu_{t,x}; \mu_1 \tvt \left\vert \bb{D}_0 (\du) - \bb{D}_0 (\nabla_x \tu )  \right\vert^2 \ril \dx\dt \\
	 &+C(\nabla_x \tu ) \frac{1}{4\epsilon} \inttom \lel  \Nu_{t,x} ;\left[ \levert \vt-\tvt \rivert^2 \right]_{\text{ess}} + \left[ 1+ \varrho e_R(\varrho,\vt ) \right]_{\text{res}} \ril \dx\dt .
\end{align*}
Similarly, we derive the following inequalities:
\begin{align*}
	&\levert \inttom \kappa(\tvt) \frac{\nabla_x \tvt}{\tvt} \cdot \lel \Nu_{t,x}; (\vt-\tvt) \left( \frac{\nabla_x \tvt}{\tvt}- \frac{\vc{D}_{\vt}}{\vt} \right) \ril \dx \dt \rivert \\
	&\leq \epsilon \inttom  \lel \Nu_{t,x} ;   \left\vert \frac{\vc{D}_\vt}{\vt} - \frac{\nabla_x \tvt}{\tvt}\right\vert^2 \ril  \!\!\dx \dt +  C(\epsilon, \tvt ) \inttom \lel \Nu_{t,x}; \levert \vt-\tvt \rivert^2 \ril \dx \dt
\end{align*}
and 
\begin{align*}
	&\levert \inttom {\nabla_x \tvt} \cdot \lel \Nu_{t,x}; (\kappa(\vr,\vt)-\kappa(\tvr,\tvt)) \left( \frac{\nabla_x \tvt}{\tvt}- \frac{\vc{D}_{\vt}}{\vt} \right) \ril \dx \dt\rivert\\
	&=\levert \inttom {\nabla_x \tvt} \cdot \lel \Nu_{t,x}; (\vt^\beta-\tvt^\beta) \left( \frac{\nabla_x \tvt}{\tvt}- \frac{\vc{D}_{\vt}}{\vt} \right) \ril \dx \dt\rivert\\
	&\leq \epsilon \inttom  \lel \Nu_{t,x} ;  (1+\vt^\beta)  \left\vert \frac{\vc{D}_\vt}{\vt} - \frac{\nabla_x \tvt}{\tvt}\right\vert^2 \ril  \!\!\dx \dt \\
	&+  C(\epsilon, \tvt ) \bigg(\inttom \lel \Nu_{t,x}; [ |\vt-\tvt |^2]_{\text{ess}} \ril \dx \dt
	 +\inttom \lel \Nu_{t,x};[1+\vt^\beta]_{\text{res}}  \ril \dx \dt \bigg).
\end{align*}
Since $ \beta \leq 2 $, the second term s controlled by radiation pressure $ p_R $.
Therefore, our choice of $ p $ yields
\begin{align*}
	 &\int_{0}^\tau \int_{ \Om} \lel  \Nu_{t,x};  \left[\vt^2 +\vert p(\vr,\vt) \vert + \vert \vc{u}-\tu \vert +  \vr \vert s(\vr, \vt) \vert\; \vert \vc{u} \vert  \right]_{\text{res}} \ril \dx\dt\\
	 &\leq  C(\epsilon)  \int_{0}^\tau \Epsilon_{mv}(t) \dt  + \epsilon \inttom \lel \Nu_{t,x}; \vert \vc{u}-\tu \vert ^2 \ril \dx \dt \\
	 &+ \int_{0}^\tau \int_{ \Om} \lel  \Nu_{t,x};  \left[   \vr \vert s_M(\vr, \vt) \vert^2\; + \vr s_R \vert \vc{u} \vert  \right]_{\text{res}} \ril \dx\dt.
\end{align*}
Furthermore, for any $ \epsilon>0 $, we get the following inequality
\begin{align}\label{E1}
	\vr s_R \vert \vc{u} \vert  \leq C(\tu) (\vt \vert \vc{u}-\tu \vert+ \varrho s_R)  \leq C(\tu) \left( \frac{1}{4\epsilon} \vt^2 + \epsilon \vert \vc{u}-\tu \vert^2 +  \vr s_R  \right).
\end{align}
On the other hand, the structural assumption on $ p_M $ and $ S_M $ gives us 
\begin{align}\label{E2}
	\vr \vert s_M(\vr, \vt) \vert^2 \leq C(1+ \varrho + \varrho e_M). 
\end{align}
Finally, we conclude that 
\begin{align*}
	&\int_{0}^\tau \int_{ \Om} \lel  \Nu_{t,x};  \left[\vt +\vert p(\vr,\vt) \vert + \vert \vc{u}-\tu \vert +  \vr \vert s(\vr, \vt) \vert\; \vert \vc{u} \vert  \right]_{\text{res}} \ril \dx\dt\\
	&\leq  C(\epsilon)  \int_{0}^\tau \Epsilon_{mv}(t) \dt  + 2 \epsilon \inttom \lel \Nu_{t,x}; \vert \vc{u}-\tu \vert ^2 \ril \dx \dt .
\end{align*}
Now proper choice of $ \epsilon $, generalized Korn-Poincar\' e inequality and Gr\" onwall's argument gives us the desired result. 

\end{proof}
\begin{Rem}
	We prove the weak strong uniqueness result for $ \kappa(\vt)= \kappa_1+\kappa_2\vt^\beta  $, $ 0\leq \beta \leq 2 $. This method does not allow us to prove the weak (measure valued)-strong uniqueness when $ \beta>2 $.
\end{Rem}
\section{Comments on measure valued solution}\label{s5}
We try to obtain a priori bounds for the system in which the pressure law follows \eqref{Pr1} with \eqref{pm,em} and \eqref{rad_p1} and the transport coefficients  $ \mu, \;\lambda $ and $ \kappa $ follow
\begin{align}\label{k,mu.lam2}
	\begin{split}
	0 < \underline{\mu} \left(1 + \vt \right) &\leq \mu(\vt) \leq \overline{\mu} \left( 1 + \vt \right),\ 
	|\mu'(\vt)| \leq c \ \mbox{for all}\ \vt \geq 0,\\
	0 &\leq  \eta(\vt) \leq \overline{\eta} \left( 1 + \vt \right), \\
	0 < \underline{\kappa} \left(1 + \vt^\beta \right) &\leq  \kappa(\vt) \leq \overline{\kappa} \left( 1 + \vt^\beta \right),\ \beta \geq 2.
\end{split}
\end{align} 
We also invoke the third law of thermodynamics \eqref{3rd law}. 
\subsection{A priori estimate}\label{appest}
Let us first recall the ballistic energy inequality
	\begin{align}\label{ap1}
		\begin{split}
		\frac{\D }{\dt} &\intO{ \left[ \frac{1}{2} \vr |\vu |^2 + \vr e - \tvt \vr s \right] } + \intO{ \frac{\tvt}{\vt}	 \left( \mathbb{S} : \bb{D}_x \vu - \frac{\vc{q} \cdot \nabla \vt }{\vt} \right) }\\
		&\leq 	\intO{ \vr \vu \cdot \vc{g}  }-  \intO{ \left[ \vr s \left( \partial_t \tvt + \vu \cdot \nabla \tvt \right) + \frac{\vc{q}}{\vt} \cdot \nabla \tvt \right] },
	\end{split}
\end{align}
where $ \tilde{\vt} $ is a smooth function  such that \[ \tvt >0 \text{ in } (0,T) \times \Om \text{ with } \tvt = \vt_B \text{ on }\partial \Om. \] 
In order to control the last integral in \eqref{ap1}, we proceed analogously as in \cite[Section 4.1] {CF2021}. Hence, we consider the extension $\tvt$ to be the unique solution of the Laplace equation 
\begin{equation} \nonumber
	\Delta \tvt (\tau, \cdot) = 0 \ \mbox{in}\ \Omega,\ \tvt(\tau, \cdot)|_{\partial \Omega} = \vt_B \ \mbox{for any}\ \tau \in [0,T].
\end{equation}
The maximum principle for the Laplace equation gives
\[
\min_{[0,T] \times \partial \Omega} \vt_B \leq \tvt (t,x) \leq \max_{[0,T] \times \partial \Omega} \vt_B \ \mbox{for any}\ (t,x) \in (0,T) \times \Omega.
\]
Let us denote this particular extension by $\widehat{\vt_B}$. 
At first we note that
\[
- \intO{ \frac{\vc{q}}{\vt} \cdot \nabla_x\widehat{\vt_B }} = \intO{ \frac{\kappa(\vt)}{\vt} \nabla_x \vt \cdot \nabla_x \widehat{\vt_B }} = 
\intO{ \nabla_x K(\vt) \cdot \nabla_x \widehat{\vt_B }} = \int_{\partial \Omega} K(\vt_B) \nabla_x \vt_B \cdot \vc{n},  
\] 
where $K'(\vt) = \frac{\kappa (\vt)}{\vt}$. 
Next, as $\widehat{\vt_B}$ is continuously differentiable in time, we obtain
\[
- \intO{ \vr s \partial_t \vt_B } \leq \left[ (1 + \intO{  \left( \frac{1}{2} \vr |\vu |^2 + \vr e - \tvt \vr s \right) } \right].
\]
Therefore, for the term
\begin{equation}  \nonumber
	- \intO{ \vr s \vu \cdot \nabla_x \widehat{\vt B}  } = - \intO{ \vr \mathcal{S} \left( \frac{\vr}{\vt^{\frac{3}{2} } } \right) \vu \cdot \nabla_x \widehat{\vt_B }} - {2a} \intO{ \vt \vu \cdot \nabla_x\widehat{\vt_B }},
\end{equation}
using the inequality in \cite[Section 4, formula (4.6)]{FN2012}, we have
\[ \varrho \left\vert \mathcal{S} \left( \frac{\vr}{\vt^{\frac{3}{2} } } \right) \right \vert  \leq \left( \vr + \vr |\log(\vr)| + \vr [ \log(\vt)]^+ \right) \ \mbox{for any} \ \vr \geq 0,\ \vt \geq 0.\]
Consequently, it yields
\begin{align*}
	\left| \intO{ \vr \mathcal{S} \left( \frac{\vr}{\vt^{\frac{3}{2} } } \right) \vu \cdot \nabla_x \widehat{\vt_B} } \right|  &\leq 
	\left( \intO{ \vr |\vu|^2 } + \intO{\vr \mathcal{S}^2\left( \frac{\vr}{\vt^{\frac{3}{2} } } \right) } \right) \\ &\leq 
	\left( 1 + 	 \intO{ \vr |\vu|^2 } + \intO{ \vr^{\frac{5}{3} } } +\intO{\vr \vt }+ \intO{\vt^2} \right)  \\
	&\leq \left( 1 +  \intO{ \left( \frac{1}{2} \vr |\vu |^2 + \vr e - \tvt \vr s \right) } \right).
\end{align*}

For the other term, we notice that
\begin{align*}
		\left| \intO{ \vt \vu \cdot \nabla \widehat{\vt_B} } \right| &\leq \epsilon \| \vu \|^2_{W^{1,2}(\Omega; R^d)} + c(\epsilon) \int_{ \Om} \vt^2 \dx .
\end{align*}
From the definition of the radiation pressure, we conclude 
\[ \int_{ \Om} \vt^2 \dx \leq C(\epsilon) \int_{ \Om}(\vr e - \tvt \vr s) \dx + c(\vt_B) + \epsilon \| \vu \|^2_{W^{1,2}(\Omega; R^d)} .\] 
Therefore, chossing $ \epsilon $ properly, we are able to get
\begin{align*} 
	&\intO{ \left( \frac{1}{2} \vr |\vu |^2 + \vr e - \tvt \vr s \right) (\tau, \cdot) }    \\
	&+ 
	\inf_{[0,T] \times \partial \Omega}\{ \vt_B \} \int_0^\tau  \intO{ \left( \| \vu \|^2_{W^{1,2}(\Omega; R^d)} + \frac{\kappa (\vt) |\nabla_x \vt|^2 }{\vt^2} \right) }\; \dt \\
	&\leq 
	\intO{ \left( \frac{1}{2} \vr_0 |\vu_0 |^2 + \vr_0 e(\vr_0, \vt_0)  - \tvt \vr_0 s (\vr_0, \vt_0) \right)  } \\
		& + c( \vt_B ) \left[ 1 +
	\int_0^\tau  \intO{ \left( \frac{1}{2} \vr |\vu |^2 + \vr e - \tvt \vr s \right)  }\; \dt \right]. 
\end{align*}
Now we can use Gr\"onwall's argument to deduce a priori estimate for $  \frac{1}{2} \vr |\vu |^2 + \vr e - \tvt \vr s $ . 

Moreover for this particular pressure law \eqref{Pr1} with \eqref{pm,em} and \eqref{rad_p1} along with the third law of thermodynamics \eqref{3rd law}, following Feireisl and B\v rezina\cite{BF2018a},  we are able to conclude that 
\begin{align*}
	&\text{ess} \sup_{(0,T)} \Vert \vr s \Vert_{L^q(\Om)} + \text{ess} \sup_{(0,T)} \Vert \vr s \vc{u}\Vert_{L^q(\Om)} \leq C 
\end{align*}
for some $ q >1$. Although they consider only for the molecular pressure, the radiation pressure case is quite straight forward. Therefore we obtain the following a priori estimate :
 \begin{align*}
	&\text{ess} \sup_{(0,T)} \int_{ \Om} \left[\vr + \vr \vert \vc{u} \vert^2  + \vr e(\vr, \vt) + |\vr \vert s(\vr,\vt)|^q \vert \right] \dx \\
	&+\int_{0}^T \int_{ \Om} \left[ \underline{\mu}\left( 1+\frac{1}{ \vt }\right) \left\vert \nabla_x \vc{u} + \nabla_x^T \vc{u} -\frac{2}{d} \Div \vc{u} \right\vert^2 + \frac{\lambda( \vt)}{2 \vt} \vert \Div \vc{u} \vert ^2\right] \dx \dt\\
	&+ \int_{0}^T \int_{ \Om} \underline{\kappa}\left(\frac{1}{\vt^2}+ \vt^{\beta-2}\right) \vert \nabla_x \vt \vert^2 \dx \dt \leq C( \vt_B),
	\end{align*}
for some $ q>1 $.
\subsection{Comments on compatibility conditions}
 Now suppose we assume a young measure $ \Nu $ is generated by a family of sequences $ (\vr_\epsilon, \vc{u}_\epsilon, \bb{D}_x \vc{u}_\epsilon, \vt_\epsilon , \nabla_x \vt_\epsilon )_{\{\epsilon>0\}}$ that satisfies the bound
 \begin{align}\label{uni-est-ep}
 	\begin{split}
 	&\text{ess} \sup_{(0,T)} \int_{ \Om} \left[\vr_\epsilon + \vr_\epsilon \vert \vc{u}_\epsilon \vert^2  + \vr_\epsilon e(\vr_\epsilon, \vt_\epsilon) + |\vr_\epsilon \vert s(\vr_\epsilon,\vt_\epsilon)|^q \vert \right] \dx \\
 	&+\int_{0}^T \int_{ \Om} \left[ \frac{\mu(\vr_\epsilon, \vt_\epsilon)}{2 \vt_\epsilon } \left\vert \nabla_x \vc{u}_\epsilon + \nabla_x^T \vc{u}_\epsilon -\frac{2}{d} \Div \vc{u}_\epsilon \right\vert^2 + \frac{\mu(\vr_\epsilon, \vt_\epsilon)}{2 \vt_\epsilon } \vert \Div \vc{u}_\epsilon \vert ^2\right] \dx \dt\\
 	&+ \int_{0}^T \int_{ \Om} \kappa(\vr_\epsilon,\vt_\epsilon) \vert \log \vt_\epsilon \vert^2 \dx \dt \leq C,
 \end{split}
 \end{align}
 uniformly with respect to $ \epsilon $ and $ \mu, \lambda \text{ and } \kappa $ follows \eqref{k,mu.lam2} and $ q>1 $. Clearly the bound \eqref{uni-est-ep} is motivated from the discussion we have in Section \ref{appest}. 
 Then following \cite[Section 2.4]{BFN2020}, we are able to deduce the velocity and temperature compatibility. 
 Also, Generalized Korn-Poincar\'e inequality can be derived in the similar way. 
 \subsubsection*{Defect measures and its compatibility}
 The bound \eqref{uni-est-ep}, gives an uniform estimate of $  \left( \frac{1}{2}\varrho_\epsilon \vert \mathbf{u}_\epsilon \vert^2 + \varrho e(\varrho_\epsilon,\vartheta_\epsilon) - \tT \varrho s(\varrho_\epsilon,\vartheta_\epsilon)\right)_{\epsilon>0}  $ with respect to $ \epsilon $ in $ L^\infty(0,T;L^1(\Om)) $.
  We denote the weak* limit of in $ L^\infty_{\text{weak-(*)}}(0,T;\mathcal{M}(\Om)) $ by  $$ \overline{\left( \frac{1}{2}\varrho \vert \mathbf{u} \vert^2 + \varrho e(\varrho,\vartheta) - \tT \varrho s(\varrho,\vartheta)\right)}. $$.Next, we define the corresponding \emph{defect measure} as
  \begin{align*}
  	\mathcal{D}_{\tT} = \overline{\left( \frac{1}{2}\varrho \vert \mathbf{u} \vert^2 + \varrho e(\varrho,\vartheta) - \tT \varrho s(\varrho,\vartheta)\right)}- \left \langle\Nutx ; \left( \frac{1}{2}\varrho \vert \mathbf{u} \vert^2 + \varrho e(\varrho,\vartheta) - \tT \varrho s(\varrho,\vartheta)\right) \right \rangle.
  \end{align*}
Now from the estimate \eqref{uni-est-ep} for entropy, we are able to conclude 
  \[ \left\langle \Nu_{t,x}; \vr s(\vr,\vt) \right\rangle= \overline{ \vr s(\vr,\vt) }   \text{for a.e.} (t,x) \in (0,T)\times \Om \] 
  where $ \overline{ \vr s(\vr,\vt) } $ is a weak limit of $ \vr_\epsilon s_\epsilon $. 
  This implies $ \mathcal{D}_{\tT}  $  is independent of $ \tT $ and eventually we have
  \begin{align*}
 	\mathcal{D} = \overline{\left( \frac{1}{2}\varrho \vert \mathbf{u} \vert^2 + \varrho e(\varrho,\vartheta) \right)}- \left \langle\Nutx ; \left( \frac{1}{2}\varrho \vert \mathbf{u} \vert^2 + \varrho e(\varrho,\vartheta) \right) \right \rangle. 
 \end{align*}
From the comparison lemma for defect measures in Feireisl et. al \cite[Lemma 2.1]{FPAW2016}, we can claim the compatibility of defects \eqref{def-m-comp}, since $ r^M $ is only related to the term $ \vr \vu \otimes \vu + p(\vr,\vt) \bb{I} $. Therefore, this particular pressure law \eqref{Pr1} allows us to have weak convergence for the terms $ \vr s $ and $ \vr s \vu $ in some $ L^q(0,T;L^q(\Om)) $, for some $ q>1 $. Hence, we avoid defect measure for entropy inequality and eventually for ballistic energy inequality. 

\section*{Concluding remarks} 
Here we briefly discuss the possible extensions and limitations of the problem:
\begin{itemize}
	\item We have given no proof of the existence of a measure valued solution. One way to prove existence is to use \emph{consistent approximations}. By consistent approximation, we mean an approximated system of the original system with error terms that vanish in the limiting case. The approximate solutions with some uniform bounds of state variables generate a Young measure, which is eventually a measure valued solution of the system. For a detailed discussion, see \cite[Chapter 5, Section 5.3.1]{FLMS2021book}.
	\item Another way to construct measure valued solution is to consider it as a weak limit of weak solutions. From \cite{CF2021}, we know that there exists a weak solution for
	\begin{align*}
		0 < \underline{\kappa} \left(1 + \vt^\beta \right) &\leq  \kappa(\vt) \leq \overline{\kappa} \left( 1 + \vt^\beta \right),\ \beta > 6,\\
		& p_R = a \vt^4, 
	\end{align*}
 Therefore, at least for system \eqref{NSF-c}-\eqref{temp_bdry} with \eqref{k,mu,lamda} and pressure following \eqref{pm,em}, \eqref{rad_p1}, we expect to obtain a measure valued solution as a suitable limit of this weak solutions. If we consider a modified radiation pressure and the heat conductivity coefficient as
\begin{align*}
	p_{\epsilon,R} = a \vt^2 + \epsilon \vt^4
\end{align*}
and 
\begin{align*}
	\kappa_\epsilon(\vt)=1+ \vt^\beta + \epsilon \vt^\gamma, \text{ with }\beta\leq 2 \text{ and }\gamma> 6.
\end{align*}
Then one can expect to generate a measure value solution by sequence of weak solutions $ \{\vr_\epsilon,\vu_\epsilon, \vt_\epsilon \}_{\{\epsilon>0\}} $. 
	\item Navier-slip boundary condition for velocity: Instead of boundary condition \eqref{velocity_bdry}, we consider the Navier-Slip boundary condition 
	\begin{align*}
	\vu \cdot \mathbf{n}|_{\partial \Om}=[\mathbb{S}(\nabla_{x}\vu)\cdot \mathbf{n}]_{\text{tan}}|_{\partial\Om}=0.
	\end{align*}
   We can provide a similar definition by modifying the Korn-Poincare inequality suitably, see \cite[Section 1.3.6]{C2021}. A similar weak(measure valued)-strong uniqueness result is expected to be true.
	\item Limitation with radiation pressure following Stefan-Boltzman law: Unfortunately, we are not able to prove our results for the radiation pressure following Stefan-Boltzman law $ p_R(\vt) = a \vt^4 $ with $ a>0 $. The main difficulty is to deal the term $\left[ \vr s_R \vert \vc{u} \vert  \right]_{\text{res}} $. Following Feireisl \cite{F2012}, we notice that, the estimation of the term needs certain Sobolev embedding which is missing in our definition. Our choice of $ p_R $ in \eqref{rad_p1} is motivated from the models of Neutron star, see Lattimer et al. \cite{LRPP1994}.
\end{itemize}

\centerline{ \bf Acknowledgement}
\vspace{2mm}

The work of N. Chaudhuri was partly supported by EPSRC Early Career Fellowship no. EP /V000586/1. The author would like to thank Prof. E. Feireisl and Dr. E. Zatorska for their valuable comments. 
\def\cprime{$'$} \def\ocirc#1{\ifmmode\setbox0=\hbox{$#1$}\dimen0=\ht0
	\advance\dimen0 by1pt\rlap{\hbox to\wd0{\hss\raise\dimen0
			\hbox{\hskip.2em$\scriptscriptstyle\circ$}\hss}}#1\else {\accent"17 #1}\fi}
		
		\bibliographystyle{abbrv}
		\bibliography{biblio_nil}
	\end{document}